\documentclass[reqno]{amsart}

\usepackage[USenglish]{babel}
\usepackage{amsfonts}
\usepackage{esint}
\usepackage[all]{xy}

\usepackage{amssymb,latexsym}
\usepackage{amsmath}
\usepackage{amsthm}
\usepackage{graphicx}
\usepackage{titletoc}
\numberwithin{equation}{section}

\renewcommand{\a }{\alpha }

\newcommand{\D }{\Delta }

\renewcommand{\l }{\lambda }

\newcommand{\Sg }{\Sigma}

\newcommand{\be}{\begin{equation}}
\newcommand{\ee}{\end{equation}}

\newcommand{\N}{\mathbb{N}}
\newcommand{\no}{\noindent}

\newcommand{\T}{\mathbb{T}}

\newtheorem{theorem}{Theorem}[section]
\newtheorem{proposition}[theorem]{Proposition}
\newtheorem{lemma}[theorem]{Lemma}
\newtheorem{corollary}[theorem]{Corollary}

\theoremstyle{definition}

\newtheorem{remark}{Remark}

\def\l{\langle}
\def\r{\rangle}

\def\XXint#1#2#3{{\setbox0=\hbox{$#1{#2#3}{\int}$}
     \vcenter{\hbox{$#2#3$}}\kern-.5\wd0}}

\def\l{\langle}
\def\r{\rangle}

\begin{document}
\title[The topological degree of the Mean field equation]{On the Topological degree of the Mean field equation \\with two parameters}

\author{ Aleks Jevnikar}

\address{ Aleks Jevnikar,~University of Rome `Tor Vergata', Via della Ricerca Scientifica 1, 00133 Roma, Italy}
\email{jevnikar@mat.uniroma2.it}

\author{Juncheng Wei}
\address{ Juncheng ~Wei,~Department of Mathematics, University of British Columbia,
Vancouver, BC V6T 1Z2, Canada}
\email{jcwei@math.ubc.ca}

\author{Wen Yang}
\address{ Wen ~Yang,~Center for Advanced Study in Theoretical Sciences (CASTS), National Taiwan University, Taipei 10617, Taiwan}
\email{math.yangwen@gmail.com}

\thanks{A.J. is partially supported by the PRIN project {\em Variational and perturbative aspects of nonlinear differential problems}. J.W. and W.Y. are partially supported by the NSERC of Canada. Part of the paper was written when A.J. was visiting the Mathematics Department at the University of British Columbia. He would like to thank the institute for the hospitality and for the financial support.}



\begin{abstract}
We consider the following class of equations with exponential nonlinearities on a compact surface $M$:
$$
  - \D u = \rho_1 \left( \frac{h_1 \,e^{u}}{\int_M
      h_1 \,e^{u} } - \frac{1}{|M|} \right) - \rho_2 \left( \frac{h_2 \,e^{-u}}{\int_M
      h_2 \,e^{-u} } - \frac{1}{|M|} \right),
$$
which is associated to the mean field equation of the equilibrium turbulence with arbitrarily signed vortices. Here $h_1, h_2$ are smooth positive
functions and $\rho_1, \rho_2$ are two positive parameters.

\medskip

We start by proving a concentration phenomena for the above equation, which leads to a-priori bound for the solutions of this problem provided $\rho_i\notin 8\pi\N, \, i=1,2$.

Then we study the blow up behavior when $\rho_1$ crosses $8\pi$ and $\rho_2 \notin 8\pi\N$. By performing a suitable decomposition of the above equation and  using the \emph{shadow system} that was introduced for the $SU(3)$ Toda system, we can compute the Leray-Schauder topological degree for $\rho_1 \in (0,8\pi) \cup (8\pi,16\pi)$ and $\rho_2 \notin 8\pi\N$.

As a byproduct our argument, we give new existence results when the underlying manifold is a sphere and a new proof for some known existence result.
\end{abstract}

\maketitle

\section{Introduction}

\no In this paper we are concerned with a mean field equation of the followin type
\begin{equation} \label{liouv}
  - \D u = \rho_1 \left( \frac{h_1 \,e^{u}}{\int_M
      h_1 \,e^{u} } - \frac{1}{|M|} \right) - \rho_2 \left( \frac{h_2 \,e^{-u}}{\int_M
      h_2 \,e^{-u} } - \frac{1}{|M|} \right) \quad \mbox{on } M,
\end{equation}
where $\D$ is the Laplace-Beltrami operator, $\rho_1, \rho_2$ are two positive parameters, $h_1,h_2$ are smooth positive functions and $M$ is a compact orientable surface without boundary. Throughout the paper, for the sake of simplicity we normalize the total volume of $M$ so that $|M|=1$.

\medskip

Equation \eqref{liouv} plays an important role in mathematical physics, as it arises as a mean field equation of the equilibrium turbulence with arbitrarily signed vortices, see Joyce and Montgomery \cite{joy-mont} and Pointin and Lundgren \cite{point-lund}. In this model the vortices are made of positive and negative intensities with the same value; here $u$ is associated with the stream function of the fluid while $\rho_1/\rho_2$ corresponds to the ratio of the numbers of the signed vortices. See, for instance, \cite{cho,lio, mar-pul,new,os} and the references therein. When the nonlinear term $e^{-u}$ in \eqref{liouv} is replaced by $e^{-\gamma u}$ with $\gamma>0$, the equation \eqref{liouv} describes a more general type of equation which arises in the context of the statistical
mechanics description of 2D-turbulence. For the recent developments of such equation, we refer the readers to \cite{pr,rt,rtzz} and the references therein. Moreover, let us point out that equation \eqref{liouv} has some connections with geometry; in fact, the case $\rho_1 = \rho_2$ turns out to be useful in the construction of constant mean curvature surfaces as explained in \cite{wen1,wen2}.

\

The goal of this paper is to compute the Leray-Schauder degree of \eqref{liouv}. To describe the main features of the problem we will first focus on the one-parameter case (when $\rho_2=0$) of \eqref{liouv}, i.e. the standard mean field equation. For future purposes, let us consider a generalization of it, in which singular sources appear on the right-hand side of \eqref{liouv}, namely
\begin{align} \label{eq2}
  - \D u = \rho \left( \frac{h \,e^{u}}{\int_M
      h \,e^{u} } - 1 \right) - 4\pi\sum_{q\in S}\alpha_q(\delta_q-1),
\end{align}
where $S$ is a finite set of points in $M$ and $\a_q \geq 0$ for all $q\in S$. Roughly speaking, in the blow up analysis of problem \eqref{liouv} when one component of $(e^u, e^{-u})$ blows up and the other one stays bounded, equation \eqref{liouv} resembles the one with singular sources \eqref{eq2}.

Equation \eqref{eq2} has a close relation with geometry as it rules the change of Gaussian curvature under conformal deformation of the underlying metric. Indeed, when $a_q = 0$ for all $q\in S$,  setting $\tilde g= e^{2v}g$ one gets
$$
	\D_{\tilde g} = e^{-2v}\D_g, \qquad	-\D_g v = K_{\tilde g} e^{2v} - K_g,
$$
where $K_g$ and $K_{\tilde g}$ are the Gaussian curvatures of $(M, g)$ and of $(M, \tilde g)$ respectively. More in general, when $a_q \neq 0$ for some $q\in S$, the new metric will have a conic singularity at the point $q$. Equation \eqref{eq2} also appears in mathematical physics in the description of the abelian or non-abelian Chern-Simons gauge field theory; we refer the interested reader to \cite{cal, du1, du2, kies, nt2}. There is an extensive literature on \eqref{eq2} in the past decades, see \cite{bah-cor,bt,bt2,cha,cha2,djadli,chenlin,cl1,cl2,cl3,djlw1,li,lin,mal,nt1,scho-zha,tar}.

One of the main difficulties in attacking this kind of problems is due to the lack of compactness; indeed, the solutions of \eqref{eq2} might blow up. This phenomena was treated in \cite{bt,bar-mon,bm, l, ls} where the authors proved a quantization result. More precisely, if we have blow up at a regular point $x_R\in M\setminus S$ for a sequence $(u_n)_n$ of solutions to \eqref{eq2}, it holds
\begin{align*}
\lim_{r \to 0} \lim_{n \to + \infty}  \rho \frac{ \int_{B_r(x_R)} h \, e^{u_n} }{ \int_M h \, e^{u_n} }  = 8 \pi.
\end{align*}
On the other hand, if blow up occurs at a singular point $x_S\in S$ with weight $4\pi\a$ then one has
\begin{align*}
\lim_{r \to 0} \lim_{n \to + \infty}  \rho \frac{ \int_{B_r(x_S)} h \, e^{u_n} }{ \int_M h \, e^{u_n} }  = 8 \pi(1+\a).
\end{align*}
Let us introduce the set $\Sigma$ of the critical parameters and it plays a crucial role for this class of equations:
\begin{align} \label{Sigma}
\begin{split}
\Sigma:&=\biggr\{8N\pi +\Sigma_{q\in A}8\pi(1+\alpha_{q})\; : \; A\subseteq S,~N\in\mathbb{N}\cup\{0\}\biggr\}\setminus\{0\}\\
&=\big\{8\mathfrak{a}_{k}\pi \; : \; k=1,2,3,\dots\big\},
\end{split}
\end{align}
where $\mathfrak{a}_{k}$ will be defined in \eqref{a_k}. By some further analysis, see for example \cite{bt,bat-man, bm, ls}, from the above quantization result it follows that the set of solutions to \eqref{eq2} is uniformly bounded in $C^{2,\beta}$ for any fixed $\beta \in (0,1)$ provided that $\rho \notin \Sg$. Thus, the Leray-Schauder degree $d_\rho$ of \eqref{eq2} is well-defined for $\rho \notin \Sg$. It was first pointed out in \cite{l} that this degree
should depend only on the topology of $M$ for the case without singularities and that $d(\rho)=1$ for $\rho < 8\pi$. Moreover, by the homotopic invariance of the degree, we have that it is a constant in each interval $(8\mathfrak{a}_k\pi, 8\mathfrak{a}_{k+1}\pi)$, where $\mathfrak{a}_0=0$. Finally, in \cite{cl1}-\cite{cl3}, Chen and Lin derived the topological degree counting formula, see Theorem A.

The numbers $\mathfrak{a}_{k}$ are combinations of the elements of the set $\Sigma$ and they can be expressed through the following \emph{generating function} $\Xi_0:$
\begin{align} \label{a_k}
\Xi_0(x)=&\bigr(1+x+x^2+x^3+\cdots\bigr)^{-\chi(M)+|S|}\Pi_{q\in S}(1-x^{1+\alpha_{q}})\nonumber\\
=&1+\mathfrak{c}_1x^{\mathfrak{a}_1}+\mathfrak{c}_2x^{\mathfrak{a}_2}+\cdots+\mathfrak{c}_kx^{\mathfrak{a}_k}+\cdots,
\end{align}
where $\chi(M)$ denotes the Euler characteristic of $M.$ Moreover, it would be helpful to define a modified generating function:
\begin{align} \label{gen}
\begin{split}
\Xi_1(x)=&\bigr(1+x+x^2+x^3+\cdots\bigr)^{-\chi(M)+1+|S|}\Pi_{q\in S}(1-x^{1+\alpha_{q}}) \\
=&1+\tilde{\mathfrak{c}}_1x^{\tilde{\mathfrak{a}}_1}+\tilde{\mathfrak{c}}_2x^{\tilde{\mathfrak{a}}_2}+\cdots+\tilde{\mathfrak{c}}_kx^{\tilde{\mathfrak{a}}_k}+\cdots.
\end{split}
\end{align}
It is easy to see that
\begin{align*}
(1+x+x^2+x^3\cdots)^{-\chi(M)+1}=1+b_1x+b_2x^2+\cdots+b_kx^k+\cdots.
\end{align*}
where
\begin{align} \label{b_k}
b_k=\left(\begin{array}{l}k-\chi(M)\\ ~\quad k\end{array}\right),
\end{align}
and
\begin{align*}
\left(\begin{array}{l}k-\chi(M)\\ ~\quad k\end{array}\right)=
\left\{\begin{array}{ll}
\frac{(k-\chi(M))\cdots(1-\chi(M))}{k!} &\mathrm{if}~k\geq1,\\
1 &\mathrm{if}~k=0.
\end{array}
\right.
\end{align*}

\medskip

With these ingredients one can express the Leray-Schauder degree for (\ref{eq2}) as stated in the following result.

\noindent {\bf{Theorem A}}. (\cite{cl3}) {\em Let $d_{\rho}$ be the Leray-Schauder degree for (\ref{eq2}), $\mathfrak{a}_k$ and $\tilde{\mathfrak{c}}_k$ be defined in \eqref{a_k} and \eqref{gen}, respectively. Suppose $8\mathfrak{a}_k\pi<\rho<8\mathfrak{a}_{k+1}\pi$. Then
$$d_{\rho}=\tilde{\mathfrak{c}}_k,$$
where $d_\rho=1$ for $k=0$.}

\

\begin{remark} \label{remark}
It is not difficult to see that when $\alpha_q = 0$ for any $q\in S$ the above formula can cover the degree counting formula for the regular case obtained in \cite{cl1} and it holds $d_{\rho}=b_k$. On the other hand, an interesting case is when $\alpha_q \in \mathbb{N}$ for any $q\in S$. In this situation the set $\Sigma$ in \eqref{Sigma} has the form $\Sigma=\{ 8n\pi \,:\, n\in\N \}$ and the generating function $\Xi_1$ in \eqref{gen} can be expressed as
\begin{align*}
\Xi_1(x)=&\bigr(1+x+x^2+x^3+\cdots\bigr)^{-\chi(M)+1}\Pi_{q\in S}\bigr(1+x+\cdots +x^{\alpha_{q}}\bigr).
\end{align*}
By direct computations, when $|S|=1$ and $\alpha_q=2$ we can get the explicit representation of $\Xi_1$ as follows
$$
\Xi_1(x)= 1+(b_1+1)x+(b_2+b_1+1)x^2+ \cdots +(b_k+b_{k-1}+b_{k-2})x^k + \cdots
$$
and it will appear in Theorem \ref{th1.5} and Theorem \ref{th1.6}.
\end{remark}

\

\noindent Concerning the mean field equation \eqref{liouv} there are fewer result regarding blow up analysis. However, one still expects an analogous quantization property to hold. This was indeed proved in \cite{jwyz} for the case $h_1=h_2=h$ by exploiting the geometric interpretation of equation \eqref{liouv} and  a quantization result concerning harmonic maps; recently, in \cite{jwy} the authors generalized this result for any choice of the two positive functions $h_1,h_2$. For a blow up point $p$ and a sequence $(u_n)_n$ of solutions to \eqref{liouv} it holds
\begin{align*}
    &\lim_{r \to 0} \lim_{n \to + \infty} \rho_{1} \frac{ \int_{B_r(p)} h \, e^{u_n} }{ \int_M h\, e^{u_n} } \in 8 \pi \N, \\
    &\lim_{r \to 0} \lim_{n \to + \infty} \rho_{2} \frac{ \int_{B_r(p)} h \, e^{-u_n} }{ \int_M h\, e^{-u_n} } \in 8 \pi \N.
\end{align*}
Let us point out that the case of multiples of $8 \pi$ indeed occurs, see \cite{ew, gp}.

\medskip

On the other hand, the topological degree theory for equation \eqref{liouv} is still not developed. Indeed, the existence result to \eqref{liouv} relies mainly on variational techniques and Morse theory, see for example \cite{m5, jev, jev2}. The only result regarding the topological degree was obtained in \cite{jev3} where the author proved that the degree is always odd provided the two parameters are comparable, namely $\rho_1,\rho_2\in (8k\pi,8(k+1)\pi), k\in\mathbb{N}$.

\

The aim of this paper is to study the blow up behavior of \eqref{liouv} in the first non-trivial case, namely when $\rho_1$ crosses $8\pi$ and $\rho_2 \notin 8\pi\N$. Then, exploiting this analysis we will provide the first degree formula for this class of equations. It is easy to see that equation \eqref{liouv} is invariant by adding constant to the solutions. Therefore, we assume that $\int_M u = 0$ and throughout the paper we will always work in the following space:
$$
	\mathring H^1 = \left\{ u\in H^1(M) \; : \; \int_M u = 0 \right\}.
$$
The first step in this program is to understand under which conditions the blow up phenomena occur. We point out that the following result can be obtained by suitably modifying the argument in \cite{bt,bat-man, bm, ls}; however, it was never proved in full details. We provide here an alternative proof, see Section \ref{sec:equiv}, which is based on a concentration property of the blowing up solutions and which is interesting by itself.
\begin{theorem}
\label{th1.1}
Suppose $h_i$ are positive smooth functions and $\rho_i\notin 8\pi\mathbb{N},~i=1,2.$ Then, there exists a positive constant $C$ such that for any solution of equation (\ref{liouv}), there holds:
$$|u(x)|\leq C \qquad \forall x\in M.$$
\end{theorem}

\medskip

\noindent It follows that the topological degree $d_{SG}$ of equation (\ref{liouv}) is well-defined for $\rho_i\notin 8\pi\mathbb{N},~i=1,2.$ Observing that by deforming the equation one gets $d_{SG} = d_{\rho_2}$ for $\rho_1<8\pi$ and $\rho_2\notin 8\pi\N$, where $d_{\rho_2}$ denotes the degree associated to equation \eqref{eq2} with $S=\emptyset$. Moreover, $d_{\rho_2}$ is known by Theorem A; therefore, our goal in this paper is to compute the degree $d_{SG}$ for $\rho_1\in(8\pi,16\pi)$. By the homotopic invariance, we can get the degree is a constant in each interval $(8k\pi, 8(k+1)\pi), k\in\N$. Then our work is reduced to calculate the difference between the degree for $\rho_1\in(0,8\pi)$ and $\rho_1\in(8\pi,16\pi)$ provided $\rho_2$ is fixed. This jump might be not zero due to the contribution by the degree of the bubbling solutions for $\rho_1$ crosses $8\pi, \rho_2 \notin 8\pi\N$, see the proof of Theorem \ref{th1.6} for more details.

For $n\in\N$ we let
\begin{equation}\label{not}
	d_{SG}(n) = d_{SG} \quad \mbox{for } \rho_1\in(8n\pi, 8(n+1)\pi),\; \rho_2\notin8\pi\N
\end{equation}
we have the following formula:
\begin{align*}
	d_{SG}(n+1)=	d_{SG}(n)~ + ~\Bigr\{\mbox{degree of the blow up solutions for }\rho_1~\mathrm{crosses}~8(n+1)\pi\Bigr\}.
\end{align*}

\medskip

In order to compute the jump of the degree we start by decomposing $u$ such that $u=v_1-v_2,$ where $v_1,v_2$ satisfy
\begin{equation}
\label{s}
\begin{cases}
\displaystyle{\Delta v_1+\rho_1\left(\frac{h_1e^{v_1-v_2}}{\int_Mh_1e^{v_1-v_2}}-1\right)}=0,~&\int_Mv_1=0,\vspace{0.2cm}\\
\displaystyle{\Delta v_2+\rho_2\left(\frac{h_2e^{v_2-v_1}}{\int_Mh_2e^{v_2-v_1}}-1\right)}=0,~&\int_Mv_2=0.
\end{cases}
\end{equation}
Concerning the equation (\ref{liouv}) and the system (\ref{s}), we have the following result, see Section \ref{sec:equiv}.
\begin{theorem}
\label{th1.2}
Let $d_{SG}$ denote the topological degree for the equation (\ref{liouv}). Then, the topological degree $d_s$ of the system (\ref{s}) is well defined and we have
$$d_{SG}=d_s \,.$$
\end{theorem}

\medskip

\noindent As a consequence of Theorem \ref{th1.2}, we can rewrite the problem (\ref{liouv}) in an equivalent way in terms of the above system (\ref{s}) such that the their degree are coincide. Then we focus on the problem \eqref{s} and calculate the degree jump when $\rho_1$ crosses $8\pi$. Specifically, we need to compute the topological degree of the bubbling solution of (\ref{s}) when $\rho_1$ crosses $8\pi,$ $\rho_2\notin 8\pi\mathbb{N}.$ Let us now introduce the Green function $G(x,p)$:
\begin{align*}
-\Delta G(x,p)=\delta_p-1~\mathrm{in}~M,\quad\mathrm{with}~\int_MG(x,p)=0.
\end{align*}
We will denote $R(x,p)$ as the regular part of the Green function $G(x,p)$. Then, we have the following result, see Section \ref{sec:shadow}.
\begin{theorem}
\label{th1.3}
Let $(v_{1k},v_{2k})$ be a sequence of solutions to (\ref{s}) with $(\rho_{1k},\rho_{2k})\rightarrow(8\pi,\rho_2)$,
and assume \mbox{$\displaystyle{\max_M}(v_{1k},v_{2k})\rightarrow\infty$.} Then, it holds:
\begin{itemize}
  \item [(i)] For some $p\in M$ we have
  \begin{align}
  \label{1.01}
  \rho_{1k}\frac{h_1e^{v_{1k}-v_{2k}}}{\int_Mh_1e^{v_{1k}-v_{2k}}}\rightarrow8\pi\delta_p\;.
  \end{align}
  \item [(ii)]
  $v_{2k}\rightarrow w~\mathrm{in}~C^{2,\alpha}(M)$: moreover the couple $(p,w)$ satisfies
\begin{align}
\label{sy}
\left\{\begin{array}{l}
\displaystyle{\Delta w+\rho_2\left(\frac{h_2e^{w-8\pi G(x,p)}}{\int_{M}h_2e^{w-8\pi G(x,p)}}-1\right)=0,}\\
\nabla\biggr(\log (h_1e^{-w})(x)+4\pi R(x,x)\biggr)_{|x=p}=0.
\end{array}\right.
\end{align}
\end{itemize}
\end{theorem}

\medskip

\noindent The system \eqref{sy} is called the \emph{shadow system} of \eqref{s}. Similar systems were obtained in \cite{lwy} and \cite{dapiru}.
We say $(p,w)$ is a non-degenerate solution of \eqref{sy} if the linearized system in $(p,w)$ admits only trivial solution. By using the transversality theorem, we will prove in Section \ref{sec:shadow} that we can always choose two positive functions $h_1, h_2$ such that the solutions of \eqref{sy} are non-degenerate. Since the topological degree is independent of $h_1,h_2$, we may assume that all the solutions of \eqref{sy} are non-degenerate and this non-degenerate property is necessary in our approach. On the contrary, for any non-degenerate solution $(p,w)$ of \eqref{sy} we can construct a sequence of bubbling solutions $(v_{1k},v_{2,k})$ of \eqref{s} with $\rho_{1k}\to8\pi,$ $\rho_2\notin 8\pi\N$ such that \eqref{sy} holds and $v_{2k}\to w.$

We will see that for a sequence of bubbling solutions which blows up a point $p$, the rate of $|\rho_{1k}-8\pi|$ plays a crucial role in all the arguments (see for example Theorem \ref{th1.4}) and it is related to the following quantity:
\begin{equation}\label{l(p)}
l(p)=\Delta\log h_1(p)-\rho_2+8\pi-2K(p),
\end{equation}
where $K(p)$ is the Gaussian curvature at $p$.
\begin{theorem}\label{bubb}
Let $(p,w)$ be a non-degenerate solution of \eqref{s}, $\rho_2\notin 8\pi\N$ and suppose $l(p)\neq 0$, where $l(p)$ is given in \eqref{l(p)}. Then, there exists a sequence of bubbling solutions $(v_{1k},v_{2k})$ to \eqref{s} with $\rho_{1k}\to8\pi$ such that \emph{(i)} and \emph{(ii)} of Theorem \ref{th1.3} hold true.
\end{theorem}

\medskip

Roughly speaking, the proof of the above result will follow by considering the solutions of \eqref{s} in the set of the bubbling solutions satisfying (i) and (ii) of  Theorem \ref{th1.3} and showing that the associated degree is not zero. More precisely, we will get the conclusion of Theorem \ref{bubb} once we prove Theorem \ref{th1.4}, see Sections \ref{sec:blowup}, \ref{sec:operator} and \ref{sec:proof}.

\medskip

Due to the presence of this kind of solutions, we need to compute the topological degree of (\ref{s}) contributed by these bubbling solutions. In particular we will see it is enough to consider the bubbling solutions contained in the subset $S_{\rho_1}(p,w)\times S_{\rho_2}(p,w)$, see \eqref{4.14} and \eqref{4.15} in Section 4 for the definition of $S_{\rho_i}(p,w),~i=1,2.$ Let $d_T(p,w)$ denote the degree contributed by the solutions $(v_{1k},v_{2k})\in S_{\rho_1}(p,w)\times S_{\rho_2}(p,w)$ and $d_S(p,w)$ denote the degree of the shadow system (\ref{sy}) contributed by the Morse index of $(p,w)$. Then we have the following result (see Section \ref{sec:proof} for the argument concerning the following results).
\begin{theorem}
\label{th1.4}
Let $\rho_2\notin 8\pi\mathbb{N}$ and suppose that $(p,w)$ is a non-degenerate solution of (\ref{sy}) and $l(p)\neq 0$, where $l(p)$ is given in \eqref{l(p)}. Let $d_T(p,w)$ and $d_S(p,w)$ be defined as above. Then
\begin{equation*}
d_T(p,w)=-\mathrm{sgn}(\rho_1-8\pi)\,d_S(p,w).
\end{equation*}
\end{theorem}

\medskip

\noindent Once we get Theorem \ref{th1.4}, it is natural for us to consider the degree of the shadow system. The idea of solving this problem is to decouple the system \eqref{sy} and then we use Theorem A to get the degree of the first equation in \eqref{sy}. The explicit result is stated in the following:
\begin{theorem}
\label{th1.5}
Assume $\rho_2\notin 8\pi\N$. Then the set of solutions $(p,w)$ for (\ref{sy}) is pre-compact in the space
$M\times\mathring{H}_1(M)$. Let $d_S$ denote the topological degree for $(\ref{sy})$. Then
\begin{equation}
\label{1.18}
d_S=\chi(M) \bigr(b_k+b_{k-1}+b_{k-2}\bigr), \qquad \rho_2\in(8k\pi,8(k+1)\pi),
\end{equation}
where $b_{-1}=b_{-2}=0.$
\end{theorem}

\medskip

\noindent Finally, by using the Theorems \ref{th1.4}, \ref{th1.5} and the fact that $d_{SG} = d_{\rho_2}$ for $\rho_1<8\pi$, which is given in Theorem A (see also Remark \ref{remark}), we can derive the following main result of the paper:
\begin{theorem}
\label{th1.6}
Let $d_{SG}$ denote the topological degree for (\ref{liouv}), then
\begin{align*}
d_{SG}=
\left\{\begin{array}{ll}
b_k,&\rho_1\in(0,8\pi),\\
b_k-\chi(M)\bigr(b_k+b_{k-1}+b_{k-2}\bigr),&\rho_1\in(8\pi,16\pi),
\end{array}\right. \qquad \rho_2\in(8k\pi,8(k+1)\pi),
\end{align*}
where $b_{-1}=b_{-2}=0.$
\end{theorem}

\

\noindent It is easy to see that when $\chi(M)\leq 0$ we can get $b_k>0$ and then $d_{SG}>0$. Therefore we can prove the following existence result in \cite{m5}.

\begin{corollary}
Let $\rho_1 \in (0,8\pi)\cup(8\pi,16\pi), \rho_2\notin 8\pi\N$ and suppose $\chi(M)\leq 0$. Then $d_{SG}>0$ and the equation \eqref{liouv} has a solution.
\end{corollary}

\

\noindent When $M$ is a sphere, we can get $d_{SG} = -1$ for $\rho_1, \rho_2\in (8\pi,16\pi)$ by direct computations. This result confirms the fact of the degree is odd stated in \cite{jev3} and gives a new proof for the following existence result in \cite{jev}.

\begin{corollary}
Let $\rho_1, \rho_2\in (8\pi,16\pi)$ and suppose $M$ is a sphere. Then $d_{SG}=-1$ and the equation \eqref{liouv} has a solution.
\end{corollary}

\

\noindent Compared with the Toda system, see \cite{lwy}, we have $d_{SG}=0$ for $\rho_1\in (8\pi,16\pi)$, $\rho_2\in (16\pi, 24\pi)$ and we can not deduce the existence of solutions to \eqref{liouv}. Furthermore, we can get a new existence result when the underlying manifold is a sphere when $\rho_1\in (8\pi,16\pi)$, $\rho_2\in (24\pi, 32\pi)$.

\begin{theorem}
Let $\rho_1\in (8\pi,16\pi)$, $\rho_2\in (24\pi, 32\pi)$ and suppose $M$ is a sphere. Then $d_{SG}=2$ and the equation \eqref{liouv} admits solutions.
\end{theorem}

\

The paper is organized as follows. In Section \ref{sec:equiv} we will prove the a-priori bound of the solutions to equation \eqref{liouv} and we will establish the degree equivalency of problems \eqref{liouv} and \eqref{s}. In Section \ref{sec:shadow} we will study the blow up phenomena for $\rho_{1k}\to 8\pi, \, \rho_2 \notin 8\pi\mathbb{N}$ and we will derive the shadow system \eqref{sy}. In Section \ref{sec:blowup} we describe the set of all the possible bubbling solutions of the equation \eqref{s}. In Section \ref{sec:operator} we use the description of the bubbling solutions obtained in Section \ref{sec:blowup} to get the leading terms of the projections associated to the degree problem. In Section \ref{sec:proof} we give the proofs for Theorems \ref{bubb}-\ref{th1.6}. Finally, in the Appendix we present some useful estimates.

\vspace{1cm}
\section{The concentration phenomenon and the equivalent formulation} \label{sec:equiv}

\medskip

In this section we will start by giving the proof of the a-priori bounds of the solutions to equation \eqref{liouv}, see Theorem \ref{th1.1}. The main ingredient will be the concentration phenomena of Lemma \ref{le2.2}. Then we will prove Theorem \ref{th1.2}, namely that the two problems \eqref{liouv} and \eqref{s} are equivalent for what concerns the degree theory.

In order to prove Theorem \ref{th1.1} we need the following preparations. For a sequence of bubbling solution $u_k$ of (\ref{liouv}), we set
\begin{align*}
&u_{k,1}=u_k-\log\int_Mh_1e^{u_k},\\
&u_{k,2}=-u_k-\log\int_Mh_2e^{-u_k}.
\end{align*}
Then, we have
$$
  - \D u_{k,1} = \rho_1 \bigr( h_1 \,e^{u_{k,1}}- 1 \bigr) - \rho_2 \bigr(h_2 \,e^{-u_{k,2}} - 1\bigr) \quad \mbox{on } M.
$$
The blow up sets for $u_{k,1}$ and $u_{k,2}$ are given by
\begin{equation}
\label{2.1}
\mathfrak{S}_i=\biggr\{p\in M\; : \; \exists\{x_k\}\subset M,~x_k\rightarrow p,~\lim_k u_{k,i}(x_k)\rightarrow+\infty\biggr\}, \quad i=1,2,
\end{equation}
and we define $\mathfrak{S}=\mathfrak{S}_1\cup\mathfrak{S}_2.$ By using the Jensen's inequality and recalling that we are working in $\mathring{H}^1$ we have
\begin{align*}
u_k=u_{k,1}+\int_Mh_1e^{u_k}\geq u_{k,1}+Ce^{\int_Mu_k}\geq u_{k,1}+C
\end{align*}
and similarly for $u_{k,2}.$ Therefore, we deduce that if $p$ is a blow up point of $u_{k,1}$ or $u_{k,2}$, then $p$ is also a blow up point of $u_k$ or $-u_k$ respectively. For any $p\in\mathfrak{S}$, we finally define the local mass by
\begin{equation}
\label{2.2}
\sigma_{p,i}=\lim_{\delta\rightarrow0}\lim_{k\rightarrow+\infty}\frac{1}{2\pi}\int_{B_{\delta}(p)}\rho_ih_ie^{u_{k,i}},
\end{equation}
which will play a crucial role in proving Theorem \ref{th1.1}. We start by observing that from the result in \cite{jwy} we have $\sigma_{p,i}\geq 4$ for some $i=\{1,2\}$ for any $p\in \mathfrak{S}$. A consequence is that $|\mathfrak{S}|<+\infty$, namely that the blow up points $\mathfrak{S}\subset M$ form a finite set. Moreover, we can prove the following result.

\medskip

\begin{lemma}
\label{le2.1}
Suppose $p\in\mathfrak{S}_i$ for some $i\in\{1,2\}$. Then $\sigma_{p,i}>0.$
\end{lemma}

\begin{proof}
We prove the lemma by contradiction. Without loss of generality we assume $p\in\mathfrak{S}_2$ and $\sigma_{p,2}=0$. We start by proving that
\begin{equation}
\label{2.3}
|u_k(x)|\leq C, \qquad \mbox{in } K\subset\subset M\setminus\mathfrak{S},
\end{equation}
for some $C>0$ depending on the set $K$. Indeed, let $M_1=\cup_{p\in\mathfrak{S}}B_{r_0}(p)$ with $r_0$ such that $K\subset\subset M\setminus M_1.$ Using the Green's representation we have
$$
u_{k}(x)=\int_MG(x,z)\Big(\rho_1(h_1e^{u_{k,1}}-1)-\rho_2(h_2e^{u_{k,2}}-1)\Big)
$$
$$
=\int_{M_1}G(x,z)\Big(\rho_1(h_1e^{u_{k,1}}-1)-\rho_2(h_2e^{u_{k,2}}-1)\Big)+\int_{M\setminus M_1}G(x,z)\Big(\rho_1(h_1e^{u_{k,1}}-1)-\rho_2(h_2e^{u_{k,2}}-1)\Big),
$$
Since $G(x,z)$ is bounded for $z\in M_1$ and $x\in K$, it is not difficult to prove that
\begin{align*}
\int_{M_1}G(x,z)\Big(\rho_1(h_1e^{u_{k,1}}-1)-\rho_2(h_2e^{u_{k,2}}-1)\Big)=O(1).
\end{align*}
On the other hand, in $M\setminus M_1$  $u_{k,i}$ are bounded above by some constant which depends on $r_0$, thus it is possible to deduce that
\begin{align*}
\int_{M\setminus {M_1}}G(x,z)\Big(\rho_1(h_1e^{u_{k,1}}-1)-\rho_2(h_2e^{u_{k,2}}-1)\Big)=O(1).
\end{align*}
This proves the claim. We point out that by the same argument one gets that $u_{k}$ has bounded oscillation in any compact subset of $M\setminus\mathfrak{S}.$ This fact will be then used in the proof of Theorem \ref{th1.1}.

\medskip

Now, by assumption $\sigma_{2p}=0$: therefore, we can take $r_0$ such that
\begin{align}
\label{2.4}
\int_{B_{r_0}(p)}\rho_2h_2e^{u_{k,2}}\leq \pi
\end{align}
for all $k$. and $r_0\leq\frac12d(p,\mathfrak{S}\setminus\{p\})$. On $\partial B_{r_0}(p)$, by (\ref{2.3})
\begin{equation}
\label{2.5}
|u_k|\leq C \qquad \mathrm{on}~\partial B_{r_0}(p).
\end{equation}
Let $\hat{u}_{1k}$ satisfy the following equation
\begin{align}
\label{2.6}
\left\{\begin{array}{ll}
\Delta \hat{u}_{1k}=\rho_1\Big(\frac{h_1e^{u_k}}{\int_Mh_1e^{u_k}}-1\Big)~&\mathrm{in}~B_{r_0}(p),\\
\hat{u}_{1k}=u_{k}~&\mathrm{on}~\partial B_{r_0}(p).
\end{array}\right.
\end{align}
We set $\hat{u}_{1k}=\hat{u}_{2k}+\hat{u}_{3k}$, where $\hat{u}_{2k}$ and $\hat{u}_{3k}$ satisfy
\begin{equation}
\label{2.7}
\left\{\begin{array}{llll}
\Delta \hat{u}_{2k}=\rho_1\frac{h_1e^{u_k}}{\int_Mh_1e^{u_{k}}}~&\mathrm{in}~B_{r_0}(p),\qquad &\hat{u}_{2k}=u_k~&\mathrm{on}~\partial B_{r_0}(p),\\
\Delta \hat{u}_{3k}=-\rho_1~&\mathrm{in}~B_{r_0}(p),\qquad &\hat{u}_{3k}=0~&\mathrm{on}~\partial B_{r_0}(p).
\end{array}\right.
\end{equation}
Exploiting the maximum principle we directly get $\hat{u}_{2k}\leq\max_{\partial B_{r_0}(p)}u_k\leq C$ by (\ref{2.5}) in $B_{r_0}(p).$ Moreover, clearly $|\hat{u}_{3k}|\leq C$ by elliptic estimates. We conclude that
\begin{equation}
\label{2.8}
\hat{u}_{1k}\leq C \qquad  \mbox{in } B_{r_0}(p).
\end{equation}

\medskip

We write now $-u_{k}=\tilde{u}_{1k}+\tilde{u}_{2k}+\hat{u}_{1k},$ where $\tilde{u}_{1k}$ and $\tilde{u}_{2k}$ satisfy
\begin{equation}
\label{2.9}
\left\{\begin{array}{llll}
\Delta \tilde{u}_{1k}=-\rho_2\frac{h_2e^{-u_k}}{\int_Mh_2e^{-u_k}}~&\mathrm{in}~B_{r_0}(p),\qquad &\tilde{u}_{1k}=0~&\mathrm{on}~\partial B_{r_0}(p),\\
\Delta \tilde{u}_{2k}=\rho_2~&\mathrm{in}~B_{r_0}(p),\qquad &\tilde{u}_{2k}=0~&\mathrm{on}~\partial B_{r_0}(p).
\end{array}\right.
\end{equation}
As before we have $|\tilde{u}_{2k}|\leq C$ in $B_{r_0}(p).$ Letting $g_k={\tilde{u}_{2k}+\hat{u}_{1k}},$  the first equation in (\ref{2.9}) can be written as
\begin{equation}
\label{2.10}
\Delta \tilde{u}_{1k}+\rho_2\frac{h_2e^{g_k}}{\int_Mh_2e^{-u_k}}e^{\tilde{u}_{1k}}=0~\mathrm{in}~B_{r_0}(p),\qquad \tilde{u}_{1k}=0~\mathrm{on}~\partial B_{r_0}(p).
\end{equation}
By Jensen's inequality and recalling that we are working in $\mathring{H}^1$ we observe that $\int_Mh_2e^{-u_k}\geq Ce^{\int_M-u_k}\geq C>0.$ Finally, setting $V_k=\rho_2\frac{h_2e^{g_k}}{\int_Mh_2e^{-u_k}},$
we have $V_k\leq C$ in $B_{r_0}(p)$, for some $C$ depending on $r_0$. Moreover, using (\ref{2.4}) we get
$$
\int_{B_{r_0}(p)}V_ke^{\tilde{u}_{1k}}\leq\pi.
$$
It follows that by \cite[Corollary 3]{bm} we have $|\tilde{u}_{1k}|\leq C$ and hence
$$-u_{k}=\tilde{u}_{1k}+\tilde{u}_{2k}+\hat{u}_{1k}\leq C \qquad \mbox{in } B_{r_0}(p).$$
We conclude that $u_{k,2}=-u_k-\int_Mhe^{-u_k}\leq C$ in $B_{r_0}(p)$, which contradicts the fact that $u_{k,2}$ blows up at $p.$ The proof is completed.
\end{proof}

\medskip

\noindent The crucial property of the blowing up sequences to \eqref{liouv}, which will be used in the proof of Theorem~\ref{th1.1}, is stated in the next lemma.
\begin{lemma}
\label{le2.2}
Suppose $u_{k,1}, u_{k,2}$ both blow up at $p\in M$ and let $$(\sigma_{p,1},\sigma_{p,2})=\big(2m(m+1),2m(m-1)\big).$$ Then
$$
u_{k,1}\rightarrow-\infty \qquad \mbox{in } B_{r_0}(p)\setminus\{p\},
$$
where $r_0$ is small enough such that $B_{r_0}(p)\cap(\mathfrak{S}\setminus\{p\})=\emptyset.$
\end{lemma}

\begin{proof}
Suppose by contradiction the claim is not true; it follows that ${u}_{k,1}>-C$ on $\partial B_{r_0}(p)$, for some $C$. Without loss of generality, we assume $m=2.$ The proof of the cases $m\geq3$ are similar. Let
$f_{1k}=\rho_1(h_1e^{u_{k,1}}-1)-\rho_2(h_2e^{u_{k,2}}-1)$ and $z_k$ be the solution of
\begin{align}
\label{2.16}
\left\{\begin{array}{ll}
-\Delta z_k=f_{1k}~&\mathrm{in}~B_{r_0}(p),\\
z_k=-C~&\mathrm{on}~\partial B_{r_0}(p).
\end{array}\right.
\end{align}
Clearly by Lemma \ref{le2.1} we have $f_{1k}\rightarrow f_1$ uniformly in any compact set of $B_{r_0}(p)\setminus\{p\}$. Moreover, by assumption we get
\begin{equation}\label{q1}
\int_{B_{r_0}(p)} f_{1k} = 16\pi+o(1) \qquad \mbox{as } r_0\rightarrow0.
\end{equation}
By the maximum principle we observe that $u_{k,1}\geq z_k$ in $B_{r_0}(p).$ It follows that
\begin{equation}\label{bou}
\int_{B_{r_0}(p)}e^{z_k}\leq \int_{B_{r_0}(p)}e^{u_{k,1}}<\infty.
\end{equation}
On the other hand, since the regular part of the Green function is bounded, by the Green's representation formula we have
\begin{align}
\label{2.17}
z_k(x)=-\int_{B_{r_0}(p)}\frac{1}{2\pi}\ln|x-y|\big(\rho_1(h_1e^{u_{k,1}}-1)
-\rho_2(h_2e^{u_{k,2}}-1)\big)+O(1).
\end{align}
For any $x\in B_{r_0}(p)\setminus\{p\}$ we let $r = \frac 12\mbox{dist}(x,p)$ and we split the above integral in the following way:
\begin{align*}
z_k(x)=&-\int_{B_{r_0}(p)}\frac{1}{2\pi}\ln|x-y|\big(\rho_1(h_1e^{u_{k,1}}-1)
-\rho_2(h_2e^{u_{k,2}}-1)\big)+O(1)\nonumber\\
=&-\int_{B_{r_0}(p)\cap B_{r}(x)}\frac{1}{2\pi}\ln|x-y|\big(\rho_1(h_1e^{u_{k,1}}-1)-\rho_2(h_2e^{u_{k,2}}-1)\big)\nonumber\\
&-\int_{B_{r_0}(p)\setminus B_r(x)}\frac{1}{2\pi}\ln|x-y|\big(\rho_1(h_1e^{u_{k,1}}-1)-\rho_2(h_2e^{u_{k,2}}-1)\big)+O(1).
\end{align*}
By the assumption ${u}_{k,i}$ are uniformly bounded above in $B_{r}(x),~i=1,2$; it follows that
$$\Big|\int_{B_{r_0}(p)\cap B_{r}(x)}\ln|x-y|\big(\rho_1(h_1e^{u_{k,1}}-1)-\rho_2(h_2e^{u_{k,2}}-1)\big)\Big|\leq C,$$
for some $C>0$ depending only on $x$.
For $y\in B_{r_0}(p)\setminus B_r(x)$, recalling \eqref{q1}, it is not difficult to see that
\begin{align*}
\int_{B_{r_0}(p)\setminus B_r(x)}\ln|x-y|\big(\rho_1(h_1e^{u_{k,1}}-1)-\rho_2(h_2e^{u_{k,2}}-1)\big)
=(16\pi+o(1))\ln|x-p|+O(1).
\end{align*}
Therefore, we get that $z_k(x)$ is uniformly bounded by some constant that depends on $x$ only. Thus, we have
$z_k\rightarrow z$ in $C^2_{loc}(B_{r_0}(p)\setminus\{p\})$, where $z$ satisfies
\begin{align*}
\left\{\begin{array}{ll}
-\Delta z=f_1~&\mathrm{in}~B_{r_0}(p)\setminus\{p\},\\
z=-C~&\mathrm{on}~\partial B_{r_0}(p).
\end{array}\right.
\end{align*}
Then, for any $\varphi\in C_0^{\infty}(B_{r_0}(p))$, by standard arguments and using \eqref{q1} one gets
\begin{align*}
\lim_{k\rightarrow+\infty}\int_{B_{r_0}(p)}\varphi(-\Delta z_k)=&\lim_{k\rightarrow+\infty} \int_{B_{r_0}(p)}(\varphi(x)-\varphi(p))(-\Delta z_k)+\varphi(p)\int_{B_{r_0}(p)}f_1+16\pi\\
=&\int_{B_{r_0}(p)}\varphi(x)f_1+16\pi\varphi(p).
\end{align*}
Thus, $-\Delta z=f_1+16\pi\delta_p$ in $B_{r_0}(p).$ Therefore, we have $z(x)\geq 8\log\frac{1}{|x-p|}+O(1)$ as $x\rightarrow p$ and we deduce
$$
\int_{B_{r_0}(p)}e^{z}=\infty,
$$
a contradiction to \eqref{bou}. The proof is concluded.
\end{proof}

\medskip

\noindent By using these lemmas we are now in position to prove the bound of the solution to equation \eqref{liouv} in Theorem~\ref{th1.1}.

\

\noindent{\em Proof of Theorem \ref{th1.1}.}
We can write (\ref{liouv}) as
$$
  - \D u_k = \rho_1 \bigr( h_1 \,e^{u_{k,1}}- 1 \bigr) - \rho_2 \bigr(h_2 \,e^{-u_{k,2}} - 1\bigr) \quad \mbox{on } M.
$$
Thus, by elliptic estimates it is  enough for us to prove that $u_{k,i}$ are uniformly bounded above. Suppose this is not true.

At first, we claim that $\mathfrak{S}_1\neq\emptyset$. If not, $u_{k,1}$ is uniformly bounded above while $u_{k,2}$ blows up. Letting $\tilde{h}_{2k}=h_2e^{-u_{1k}}$ we write $u_{k}=u_{1k}-u_{2k}$, where $u_{1k}$ and $u_{2k}$ satisfies
\begin{align*}
\begin{cases}
\Delta u_{1k}+\rho_1(h_1e^{u_{k,1}}-1)=0,\quad &\int_Mu_{1k}=0,\\
\Delta u_{2k}+\rho_2\Big(\frac{\tilde{h}_{2k}e^{u_{2k}}}{\int_M \tilde{h}_{2k}e^{u_{2k}}}-1\Big)=0,~&\int_Mu_{2k}=0.
\end{cases}
\end{align*}
By the $L^p$ estimate, $u_{1k}$ is bounded in $W^{2,p}$ for any $p>1$ and hence we deduce that $u_{1k}$ is bounded in $C^{1,\alpha}$ for any $\alpha\in(0,1)$. After eventually passing to a subsequence, $u_{1k}$ converges to $\hat{u}_{1}$ in $C^{1,\alpha}.$ We conclude that $\tilde{h}_{2k}\rightarrow h_2e^{-\hat{u}_{1}}$ in $C^{1,\alpha}$. By the fact that $u_{k,2}$ blows up, also $u_{2k}$ blows up. Therefore, we are in position to applying the result in \cite{ls} and get
$\rho_2\in 8\pi\mathbb{N}$, which contradicts our assumption. Thus $\mathfrak{S}_1\neq\emptyset$. Similarly,  $\mathfrak{S}_2\neq\emptyset$.

\medskip

By similar arguments one can show that $\mathfrak{S}_1\cap\mathfrak{S}_2\neq\emptyset.$ In fact, suppose
$\mathfrak{S}_1\cap\mathfrak{S}_2=\emptyset.$ Let $p\in \mathfrak{S}_2,$ and take $r_0$  small enough such that $B_{r_0}(p)\cap(\mathfrak{S}\setminus\{p\})=\emptyset.$
Letting  $\hat{h}_{2k}=h_2e^{u_{3k}}$ we decompose
$u_{k,2}=u_{3k}+u_{4k}$ such that $u_{3k}$ and $u_{4k}$ satisfy
\begin{align}
\label{2.11}
\left\{\begin{array}{llll}
\Delta u_{3k}-\rho_1(he^{u_{k,1}}-1)=0~&\mathrm{in}~B_{r_0}(p),\quad
&u_{3k}=0~&\mathrm{on}~\partial B_{r_0}(p),\\
\Delta u_{4k}+\rho_2(\hat{h}_{2k}e^{u_{4k}}-1)=0~&\mathrm{in}~B_{r_0}(p),\quad
&u_{4k}=u_{k,2}~&\mathrm{on}~\partial B_{r_0}(p).
\end{array}\right.
\end{align}
By construction $u_{k,1}$ is uniformly bounded from above in $B_{r_0}(p)$, hence $\hat{h}_{2k}$ converges to some function
in $C^{1,\alpha}(B_{r_0}(p)).$ On the other hand, $u_{4k}$ blows up at $p$ and by the result in \cite{l} we get
\begin{align}
\label{2.12}
\left|u_{4k}-\log\left(\frac{e^{u_{4k}(p^{(k)})}}
{\left(1+\frac{\rho_2\hat{h}_{2k}(p^{(k)})e^{u_{4k}(p^{(k)})}}{8}|x-p^{(k)}|^{2}\right)^2}\right)\right|\leq C \qquad \mbox{in } B_{r_0}(p),
\end{align}
where $u_{4k}\bigr(p^{(k)}\bigr)=\max_{B_{r_0}(p)}u_{4k}$. From (\ref{2.12}) and by the definition of $\hat{h}$ we have
\begin{align}
\label{2.13}
\begin{split}
&u_{4k}\rightarrow-\infty \qquad \quad \mathrm{in}~ B_{r_0(p)}\setminus\{p\}, \\
&\rho_2h_2e^{u_{k,2}}\rightarrow 8\pi\delta_p \qquad  \mathrm{in}~B_{r_0}(p).
\end{split}
\end{align}
As a consequence, we have
\begin{equation}
\label{2.14}
\rho_2=\lim_{k\rightarrow\infty}\int_M\rho_2h_2e^{u_{k,2}}=8\pi|\mathfrak{S}_2|,
\end{equation}
a contradiction to our assumption $\rho_2\notin 8\pi\mathbb{N}$, so $\mathfrak{S}_1\cap\mathfrak{S}_2\neq\emptyset.$
\medskip

Let $p\in \mathfrak{S}_1\cap\mathfrak{S}_2,$ and $\sigma_{p,i},i=1,2$ be the local masses. Applying the result in \cite{jwy} we deduce
$$(\sigma_{p,1},\sigma_{p,2})=\big(2m(m-1),2m(m+1)\big) \quad \mathrm{or} \quad \big(2m(m+1),2m(m-1)\big)$$
for some integer $m>0.$ By Lemma \ref{le2.1} we have $m\geq2.$ Suppose for example that $\sigma_{p,1}=2m(m+1)$ and $\sigma_{p,2}=2m(m-1)$. Applying Lemma \ref{le2.2} we have that $u_{k,1}$ concentrates, i.e. $u_{k,1}\rightarrow-\infty$ uniformly in any compact set of $B_{r_0}(p)\setminus\{p\}.$  Since $u_{k,1}$ has bounded oscillation outside the blow up set, see the argument after \eqref{2.3}, it follows $u_{k,1}\rightarrow-\infty$ uniformly in any compact set of $M\setminus\mathfrak{S}_1.$ Reasoning as before we get
\begin{equation}
\label{2.15}
\rho_1h_1e^{u_{k,1}}\rightarrow\alpha_q\sum_{q\in\mathfrak{S}_1\setminus\{p\}}\delta_q
+4\pi m(m+1)\delta_p \qquad \mathrm{with}~\alpha_q=8n\pi~\mathrm{for~some~integer}~n,
\end{equation}
which implies $\rho_1\in8\pi\mathbb{N}$ and thus we get a contradiction. Following exactly the same process we can also get a contradiction if $(\sigma_{p,1},\sigma_{p,2})=\big(2m(m-1),2m(m+1)\big)$. This concludes the proof of the Theorem~\ref{th1.1}.
\begin{flushright}
$\square$
\end{flushright}

\

The second part of this section is concerned with the proof of Theorem \ref{th1.2}, that is the equivalence of the topological degree between the equation \eqref{liouv} and the system
\begin{align}
\label{2.19}
\begin{cases}
\Delta v_1+\rho_1\Big(\frac{h_1e^{v_1-v_2}}{\int_Mh_1e^{v_1-v_2}}-1\Big)=0,\quad&\int_Mv_1=0,\\
\Delta v_2+\rho_2\Big(\frac{h_2e^{v_2-v_1}}{\int_Mh_2e^{v_2-v_1}}-1\Big)=0,\quad&\int_Mv_2=0,
\end{cases}
\end{align}
where $u=v_1-v_2.$

\

\noindent{\em Proof of Theorem \ref{th1.2}.}
By Theorem \ref{th1.1} we have  that the degree $d_{SG}$ is well defined for $\rho_1,\rho_2\notin 8\pi\mathbb{N}.$ As we discussed in the introduction, we  decompose $u=v_1-v_2$, where $v_1,v_2$ satisfy \eqref{2.19}.

It is easy to see that such decomposition is unique, i.e. for a given solution $u$ to (\ref{liouv}) we can find a unique
solution $(v_1,v_2)$ to (\ref{2.19}) such that $u=v_1-v_2.$ On the other hand, if $(v_1,v_2)$ is a solution to (\ref{2.19}), $u=v_1-v_2$ is automatically a solution to (\ref{liouv}).

\medskip

Consider now a solution $(v_1,v_2)$ to (\ref{2.19}). We start by observing that since $v_1-v_2$ is a solution to \eqref{liouv}, it is bounded for $\rho_i\notin 8\pi\N, i=1,2$, see Theorem \ref{th1.1}. Moreover, notice that $\int_M e^{v_1-v_2}\geq C>0$ by Jensen's inequality. Therefore, by elliptic estimates we get also $v_1,v_2$ are bounded provided $\rho_i\notin 8\pi\N, i=1,2$. It follows that the topological degree $d_s$ of the system \eqref{2.19} is well defined in this range of parameters.

\medskip

Next, we prove that the Morse index of a solution $u$ to (\ref{liouv}) is exactly the same as the Morse index of the corresponding solution $(v_1,v_2)$ to (\ref{2.19}). We consider the linearized equation of (\ref{liouv}) at $u,$
\begin{align*}
L(\phi)=&\Delta\phi+\rho_1\frac{h_1e^u}{\int_Mh_1e^u}\phi-\rho_1\frac{h_1e^u}{(\int_Mh_1e^u)^2}\int_Mh_1e^u\phi\\
&+\rho_2\frac{h_2e^{-u}}{\int_Mh_2e^{-u}}\phi-\rho_2\frac{h_2e^{-u}}{(\int_Mh_2e^{-u})^2}\int_Mh_2e^{-u}\phi.
\end{align*}
If $\phi$ is an eigenfunction of the linearized operator of $L$ with negative eigenvalue $\lambda$, i.e.
\begin{align}
\label{2.20}
L(\phi)+\lambda\phi=0 \qquad \mathrm{with}~\lambda<0,
\end{align}
we decompose $\phi$ as $\phi=\phi_1-\phi_2,$ where $\phi_1$ and $\phi_2$ satisfy
\begin{align}
\label{2.21}
&\Delta\phi_1+\rho_1\frac{h_1e^u}{\int_Mh_1e^u}(\phi_1-\phi_2)
-\rho_1\frac{h_1e^u}{(\int_Mh_1e^u)^2}\int_Mh_1e^u(\phi_1-\phi_2)+\lambda\phi_1=0,\nonumber\\
&\Delta\phi_2+\rho_2\frac{h_2e^{-u}}{\int_Mh_2e^{-u}}(\phi_2-\phi_1)
-\rho_2\frac{h_2e^{-u}}{(\int_Mh_2e^{-u})^2}\int_Mh_2e^{-u}(\phi_2-\phi_1)+\lambda\phi_2=0.
\end{align}
In the following, we claim there is a map between $\phi$ and such $(\phi_1,\phi_2)$ and this map is one to one. Indeed, for any function $\phi$ and parameter $\lambda$ negative, we consider the following system:
\begin{align}
\label{2.22}
&(\Delta+\lambda)\phi_1+\rho_1\frac{h_1e^u}{\int_Mh_1e^u}\phi
-\rho_1\frac{h_1e^u}{(\int_Mh_1e^u)^2}\int_Mh_1e^u\phi=0,\nonumber\\
&(\Delta+\lambda)\phi_2-\rho_2\frac{h_2e^{-u}}{\int_Mh_2e^{-u}}\phi
+\rho_2\frac{h_2e^{-u}}{(\int_Mh_2e^{-u})^2}\int_Mh_2e^{-u}\phi=0.
\end{align}
For fixed $\phi$ and $\lambda$ negative, we can always solve (\ref{2.22}) and get a unique solution $(\phi_1,\phi_2).$ While for any solution $(\phi_1,\phi_2)$ of (\ref{2.22}) and $\lambda$ negative, $\phi=\phi_1-\phi_2$ is automatically a solution of (\ref{2.20}). Hence, we prove the claim.

On the other hand, we can see that (\ref{2.21}) is nothing but the linearized equation of (\ref{2.19}) at $(v_1,v_2)$. Therefore, the Morse index of the solution $u$ to (\ref{liouv})  is same as the Morse index of the solution $(v_1,v_2)$ to (\ref{2.19}). According to the definition of the topological degree and since the decomposition for $u$ is unique, we can conclude that the topological degree of these two equations are the same.
\begin{flushright}
$\square$
\end{flushright}

\vspace{1cm}
\section{shadow system} \label{sec:shadow}

\medskip

The first goal of this section is to provide a proof of Theorem \ref{th1.3} by studying the blow up phenomena for $\rho_{1k}\rightarrow 8\pi,\rho_{2k}\rightarrow\rho_2$ with $\rho_2\notin 8\pi\mathbb{N}$. In the second part of the section we will prove that it is possible to choose $h_1,~h_2$ such that the associated shadow system is non-degenerate.

\

\noindent{\em Proof of Theorem \ref{th1.3}.} Let $\rho_2\notin 8\pi\N$ and $(\rho_{1k},\rho_{2k})\to(8\pi,\rho_2)$. Consider a sequence of solutions $(v_{1k},v_{2k})$ to (\ref{s}) such that $\max_M(v_{1k},v_{2k})\to+\infty$.

We claim that $v_{2k}$ converges to some function $w$ in $C^{1,\alpha}(M)$ (passing to a subsequence if necessary) and that $v_{1k}$ blows up at only one point.

Indeed, we have that $u_k=v_{1k}-v_{2k}$ is a solution of (\ref{liouv}). Then, from the proof of Theorem \ref{th1.1}, $-u_k=v_{2k}-v_{1k}$ is uniformly bounded above: by using Jensen's inequality and the classical elliptic estimates, from the second equation in (\ref{s}) we conclude $v_{2k}$ is uniformly bounded in $C^{1,\alpha}$ and hence the first part of the claim is proved. As a consequence we  deduce that $\max_Mv_{1k}\rightarrow+\infty.$ Furthermore, noticing that $\rho_1\rightarrow8\pi,$ $v_{1k}$ blows up at only one point, say $p\in M.$
\medskip

We write the first equation in (\ref{s}) as
\begin{equation}
\label{3.1}
\Delta v_{1k}+\rho_{1k}\left(\frac{\tilde{h}_{1k}e^{v_{1k}}}{\int_{M}\tilde{h}_{1k}e^{v_{1k}}}-1\right)=0,
\end{equation}
where $\tilde{h}_{1k}=h_1e^{-v_{2k}}$. We define $\hat{v}_{1k}=v_{1k}-\log\int_{M}h_ke^{v_{1k}}.$ Due to the $C^{1,\alpha}$ convergence of $\tilde{h}_{1k},$ we can apply the following result of Li \cite{l}:
\begin{align}
\label{3.2}
\left|\hat{v}_{1k}-\log\frac{e^{\lambda_{k}}}{\left(1+\frac{\rho_{1k}\tilde h_{1k}(p^{(k)})
e^{\lambda_{k}}}{8}|x-p^{(k)}|^2\right)^2}\right|<c \qquad
\mathrm{for}~|x-p^{(k)}|<r_0,
\end{align}
where $\lambda_{k}=\hat{v}_{1k}(p^{(k)})=\max_{x\in B_{r_0}(p)}\hat{v}_{1k}.$ It follows that
\begin{equation}
\label{3.3}
\hat{v}_{1k}\rightarrow-\infty~\mathrm{in}~M\setminus\{p\}\qquad \mbox{and} \qquad
\rho_{1k}\frac{h_1e^{v_{1k}-v_{2k}}}{\int_Mh_1e^{v_{1k}-v_{2k}}}\rightarrow 8\pi\delta_p,
\end{equation}
This conclude the first part of Theorem \ref{th1.1}. Moreover, in this setting the authors in \cite{chenlin} proved that
\begin{align}
\label{3.4}
\nabla\big(\log(h_1e^{-w})+4\pi R(x,x)\big)\mid_{x=p}=0,
\end{align}
which gives the second equation in (\ref{sy}).
\medskip

Reasoning as in Lemma 2.4 in \cite{lwy}, one can use the following property in \cite{chenlin}:
$$
\left|\nabla \hat{v}_{1k}-\nabla \left(\log\frac{e^{\lambda_{k}}}{\left(1+\frac{\rho_{1k}\tilde h_{1k}(p^{(k)})
e^{\lambda_{k}}}{8}|x-p^{(k)}|^2\right)^2}\right)\right|<c \qquad
\mathrm{for}~|x-p^{(k)}|<r_0
$$
and get $v_{2k}\rightarrow w$ in $C^{2,\alpha}(M)$. From this convergence we are in position to apply a result in \cite{l}, which asserts that
$$v_{1k}\rightarrow 8\pi G(x,p) \qquad \mathrm{in}~C^{2,\alpha}(M\setminus\{p\}).$$
The above convergence jointly with $v_{2k}\rightarrow w$ in $C^{2,\alpha}(M)$ yields that $w$ satisfies the following equation:
\begin{equation}
\label{3.5}
\Delta w+\rho_2\left(\frac{h\,e^{w-8\pi G(x,p)}}{\int_M h\,e^{w-8\pi G(x,p)}}-1\right)=0.
\end{equation}
This proves the first equation in (\ref{sy}). Therefore, we finish the proof of Theorem \ref{th1.2}.
\begin{flushright}
$\square$
\end{flushright}

\

\noindent The second part of this section is devoted in showing the non-degeneracy of the shadow system (\ref{sy}), see Proposition \ref{th2.1.2}. This will be carried out by applying the well-known transversality theorem, which can be found \cite{a,q} and the references therein. Although we can suitably adapt the argument in \cite{lwy}, for the sake of completeness we give the details here.

\medskip

First, we give some notations. Let $\mathcal{H},~\mathcal{B}$ and $\mathcal{E}$ be  Banach manifolds with
$\mathcal{H}$ and $\mathcal{E}$ separable. Let $F:\mathcal{H}\times \mathcal{B}\rightarrow \mathcal{E}$ be a $C^k$ map. We say $y\in \mathcal{E}$ is a regular value if every point $x\in F^{-1}(y)$ is a regular point;
$x\in \mathcal{H}\times \mathcal{B}$ is a regular point of $F$ if $D_xF:T_x(\mathcal{H}\times \mathcal{B})\rightarrow T_{F(x)}\mathcal{E}$ is onto. We say a set $A$ is a residual set in $\mathcal{B}$ if $A$ is a countable intersection of open dense sets in $\mathcal{B}$, which implies in particular that $A$ is dense in $\mathcal{B}$ ($\mathcal{B}$ is a Banach space), see \cite{k}.
\begin{theorem} [\cite{a,q}]
\label{th2.1.1}
Let $\mathcal{H},~\mathcal{B},~\mathcal{E}$ be as above and let $F:\mathcal{H}\times \mathcal{B}\rightarrow \mathcal{E}$ be a $C^k$ map. If $0$ is a regular value of $F$ and $F_b=F(\cdot,b)$ is a Fredholm map of index less than $k,$ then the set
$$\Bigr\{b\in \mathcal{B}:0~is~a~regular~value~of~F_b\Bigr\}$$ is residual in $\mathcal{B}.$ In particular, the above set is dense in $\mathcal{B}.$
\end{theorem}

\medskip

With this in hand we can now prove the following result, which will be used crucially in the sequel when we construct approximate blow up solutions to \eqref{s}.
\begin{proposition}\label{th2.1.2}
There exist $h_1,h_2$ positive smooth functions such that the solutions to the shadow system \eqref{sy} are non-degenerate.
\end{proposition}
\begin{proof}
Following the notations in Theorem \ref{th2.1.1}, we denote
$$\mathcal{H}=\mathring{W}^{2,\mathfrak{p}}(M)\times M, \quad \mathcal{B}=C^{2,\alpha}(M)\times C^{2,\alpha}(M),
\quad \mathcal{E}=\mathbb{R}^2\times\mathring{W}^{0,\mathfrak{p}}(M),$$
where
$$\mathring{W}^{2,\mathfrak{p}}(M):=\left\{f\in W^{2,\mathfrak{p}}\mid \int_Mf=0\right\}, \quad
\mathring{W}^{0,\mathfrak{p}}(M):=\left\{f\in L^{\mathfrak{p}}\mid \int_Mf=0\right\}.$$

\noindent We consider the map
\begin{equation}
\label{3.6}
T(w,p,h_1,h_2)=\left[\begin{array}{l}
\Delta w+\rho_2\left(\frac{h_2e^{w-8\pi G(x,p)}}
{\int_Mh_2e^{w-8\pi G(x,p)}}-1\right)\vspace{0.2cm}\\
\nabla\log\big(h_1e^{-w}+4\pi R(x,x)\big)(p)
\end{array}\right].
\end{equation}
Clearly, $T$ is $C^1.$ In order to apply Theorem \ref{th2.1.1}: we have to prove that
\begin{enumerate}
\item [(i)] $T(\cdot,\cdot,h_1,h_2)$ is a Fredholm map of index $0$,
\item [(ii)] $0$ is a regular value of $T.$
\end{enumerate}

\medskip

\noindent We start by proving (i). We have
\begin{equation}
\label{3.7}
T'_{w,p}(w,p,h_1,h_2)[\phi,\nu]=\left[\begin{array}{l}
T_0(w,p,h_1,h_2)[\phi,\nu]\\
T_1(w,p,h_1,h_2)[\phi,\nu]
\end{array}\right],
\end{equation}
where
\begin{align*}
T_0(w,p,h_1,h_2)[\phi,\nu]=~&\Delta\phi+\rho_2\frac{h_2e^{w-8\pi G(x,p)}}
{\int_Mh_2e^{w-8\pi G(x,p)}}\phi\\
&-\rho_2\frac{h_2e^{w-8\pi G(x,p)}}{(\int_Mh_2e^{w-8\pi G(x,p)})^2}\int_Mh_2e^{w-8\pi G(x,p)}\phi\\
&-8\pi\rho_2\frac{h_2e^{w-8\pi G(x,p)}}{\int_Mh_2e^{w-8\pi G(x,p)}}\nabla G(x,p)\cdot\nu\\
&+8\pi\rho_2\frac{h_2e^{w-8\pi G(x,p)}}{(\int_Mh_2e^{w-8\pi G(x,p)})^2}\int_Mh_2e^{w-8\pi G(x,p)}\nabla G(x,p)\cdot\nu,
\end{align*}
\begin{align*}
T_1(w,p,h_1,h_2)[\phi,\nu]=&\nabla^2_x\big(\log h_1e^{-w}+4\pi R(x,x)\big)\mid_{x=p}\cdot \nu-\nabla\phi(p).
\end{align*}
The idea now is to decompose the map in the following way:
\begin{align}
\label{3.8}
T'_{w,p}[\phi,\nu]=\left[\begin{array}{l}
T_{01}\\
T_{11}
\end{array}\right][\phi,\nu]+
\left[\begin{array}{l}
T_{02}\\
T_{12}
\end{array}\right][\phi,\nu],
\end{align}
where
$$T_{11}=0,~T_{12}=T_1,$$
\begin{align*}
T_{01}(w,p,h_1,h_2)[\phi,\nu]=&\Delta\phi+\rho_2\frac{h_2e^{w-8\pi G(x,p)}}
{\int_Mh_2e^{w-8\pi G(x,p)}}\phi\\
&-\rho_2\frac{h_2e^{w-8\pi G(x,p)}}
{(\int_Mh_2e^{w-8\pi G(x,p)})^2}\int_Mh_2e^{w-8\pi G(x,p)}\phi,
\end{align*}
and
\begin{align*}
T_{02}(w,p,h_1,h_2)[\phi,\nu]=&-8\pi\rho_2\frac{h_2e^{w-8\pi G(x,p)}}
{\int_Mh_2e^{w-8\pi G(x,p)}}\nabla G(x,p)\nu\\
&+8\pi\rho_2\frac{h_2e^{w-8\pi G(x,p)}}
{(\int_Mh_2e^{w-8\pi G(x,p)})^2}\int_Mh_2e^{w-8\pi G(x,p)}\nabla G(x,p)\nu.
\end{align*}
Let $\mathfrak{T}_1=\left[\begin{array}{l}
T_{01}\\
T_{11}
\end{array}\right]$ and $\mathfrak{T}_2=\left[\begin{array}{ll}
T_{02}\\
T_{12}
\end{array}\right].$ In this way we notice that $\mathfrak{T}_1$ is a symmetric operator; it follows that $\mathfrak{T}_1$
is a Fredholm operator of index $0.$ Combining the Sobolev inequality and the fact that $\mathbb{R}^2$ is a finite Euclidean space, it is possible to show that
$\mathfrak{T}_2$ is a compact operator. Therefore, by the operator theory, see for example \cite{k}, we get $\mathfrak{T}_1+\mathfrak{T}_2$ is
also a Fredholm linear operator with index $0$. We conclude that $T$ is a Fredholm map with index $0$ and (i) is proved.

\

We are left with the proof of (ii), i.e. that $0$ is a regular value. One gets
\begin{align*}
T_{h_1}'(w,p,h_1,h_2)[H_1]=
\left[\begin{array}{l}
\quad\quad\quad\quad0\\
\frac{\nabla H_1}{h_1}(p)-\frac{\nabla h_1}{(h_1)^2}H_1(p)
\end{array}\right],
\end{align*}
and
\begin{align*}
&T_{h_2}'(w,p,h_1,h_2)[H_2]\\
&=\left[\begin{array}{l}
\rho_2\frac{H_2e^{w-8\pi G(x,p)}}{\int_Mh_2e^{w-8\pi G(x,p)}}-
\rho_2\frac{h_2e^{w-8\pi G(x,p)}}{(\int_Mh_2e^{w-8\pi G(x,p)})^2}\int_MH_2e^{w-8\pi G(x,p)}\\
\quad\quad\quad\quad\quad\quad\quad\quad\quad\quad\quad\quad0
\end{array}\right].
\end{align*}
By choosing $\nu=0$ and $H_1$ such that $\frac{\nabla H_1}{h_1}(p)-\frac{\nabla h_1}{(h_1)^2}H_1(p)=\nabla\phi$ we obtain
\begin{align*}
&T'_{w,p}(w,p,h_1,h_2)[\phi,\nu]+T_{h_1}'(w,p,h_1,h_2)[H_1]\\
&=\left[\begin{array}{l}
\Delta\phi+\rho_2\frac{h_2e^{w-8\pi G(x,p)}}{\int_Mh_2e^{w-8\pi G(x,p)}}\phi-
\rho_2\frac{h_2e^{w-8\pi G(x,p)}}{(\int_Mh_2e^{w-8\pi G(x,p)})^2}\int_Mh_2e^{w-8\pi G(x,p)}\phi\\
\quad\quad\quad\quad\quad\quad\quad\quad\quad\quad\quad\quad\quad\quad0
\end{array}\right],
\end{align*}
which is a symmetric operator and hence a Fredholm operator of index $0.$ For this choice of $\nu$ and $H_1$, we claim that
$$
\left[\begin{array}{l}f\\0\end{array}\right] \;\subset\; \Bigr(T'_{w,p}(w,p,h_1,h_2)[\phi,\nu] + T_{h_1}'(w,p,h_1,h_2)[H_1]\Bigr) + T_{h_2}'(w,p,h_1,h_2)[H_2],
$$
for all $f\in \mathring{W}^{0,p}.$ One can observe that it is enough to prove that only $\phi=0$ can satisfy
$$\phi\in\mathrm{Ker}\left\{\Delta\cdot+\rho_2\frac{h_2e^{w-8\pi G(x,p)}}{\int_Mh_2e^{w-8\pi G(x,p)}}\cdot
-\rho_2\frac{h_2e^{w-8\pi G(x,p)}}{(\int_Mh_2e^{w-8\pi G(x,p)})^2}\int_Mh_2e^{w-8\pi G(x,p)}\cdot\right\}$$
and
$$\Big\l\phi,~\rho_2\frac{H_2e^{w-8\pi G(x,p)}}{\int_Mh_2e^{w-8\pi G(x,p)}}-
\rho_2\frac{h_2e^{w-8\pi G(x,p)}}{(\int_Mh_2e^{w-8\pi G(x,p)})^2}\int_MH_2e^{w-8\pi G(x,p)}\Big\r=0,$$
for all $H_2\in C^{2,\alpha}(M).$ The latter property can be rewritten as
$$\Big\l\phi,~\rho_2\frac{h_2e^{w-8\pi G(x,p)}}{\int_Mh_2e^{w-8\pi G(x,p)}}\frac{H_2}{h_2}-\rho_2\frac{h_2e^{w-8\pi G(x,p)}}
{(\int_Mh_2e^{w-8\pi G(x,p)})^2}\int_Mh_2e^{w-8\pi G(x,p)}\frac{H_2}{h_2}\Big\r=0.$$
We set now
$$L=\Delta\cdot+\rho_2\frac{h_2e^{w-8\pi G(x,p)}}{\int_Mh_2e^{w-8\pi G(x,p)}}\cdot
-\rho_2\frac{h_2e^{w-8\pi G(x,p)}}{(\int_Mh_2e^{w-8\pi G(x,p)})^2}\int_Mh_2e^{w-8\pi G(x,p)}\cdot\;.$$
Since $\phi\in\mathrm{Ker}(L)$ we have that
\begin{align}
\label{3.9}
\int_{M}L(\phi)\cdot\overline{H}_2=0, \qquad \forall \overline{H}_2\in W^{0,\mathfrak{p}}(M).
\end{align}
On the other hand $C^{2,\alpha}(M)$ is dense in $W^{0,\mathfrak{p}}(M)$ and
$$\Big\l\phi,\rho_2\frac{h_2e^{w-8\pi G(x,p)}}{\int_Mh_2e^{w-8\pi G(x,p)}}\overline{H}_2-\rho_2\frac{h_2e^{w-8\pi G(x,p)}}
{(\int_Mh_2e^{w-8\pi G(x,p)})^2}\int_Mh_2e^{w-8\pi G(x,p)}\overline{H}_2\Big\r=0,$$
therefore, we get
\begin{equation}
\label{3.10}
\int_M\Delta\phi\cdot\overline{H}_2=0,\qquad \forall~\overline{H}_2\in W^{0,\mathfrak{p}}(M).
\end{equation}
It follows that
\begin{equation}
\label{3.11}
\Delta\phi=0\quad \mathrm{in}~M,\qquad\int_M\phi=0,
\end{equation}
which yields $\phi\equiv0.$  Thus the claim holds true.
\medskip

On the other hand, one can find two functions $H_{1,1}$ and $H_{1,2}$ such that
$$\frac{\nabla H_{1,1}}{h_1}(p)-\frac{\nabla h_1}{(h_1)^2}H_{1,1}(p)=(1,0),$$
and
$$\frac{\nabla H_{1,2}}{h_1}(p)-\frac{\nabla h_1}{(h_1)^2}H_{1,2}(p)=(0,1).$$
Then it is not difficult to see that
$$\left[\begin{array}{l}0\\c\end{array}\right]\subset DT(w,p,h_1,h_2)[\phi,\nu],$$
for all $c\in \mathbb{R}^2.$ This concludes the proof that the differential map is onto. Hence we get (ii), i.e. that $0$ is a regular point of $T.$

\

Applying Theorem \ref{th2.1.1} we have that
$$\Big\{(h_1,h_2)\in \mathcal{B}:0~\mathrm{is~a~regular~value~of}~T(\cdot,\cdot,h_1,h_2)\Big\}$$
is residual in $\mathcal{B}.$ Since $T(w,p,h_1,h_2)$ is a Fredholm map of index $0$ for fixed $h_1,h_2,$ we have
$$\Big\{(h_1,h_2)\in \mathcal{B}:~\mathrm{the~solution}~(w,p)~\mathrm{of}~T(\cdot,\cdot,h_1,h_2)=0~\mathrm{is~non\mbox{-}degenerate}\Big\}$$
is residual in $\mathcal{B}.$ In particular, it is dense in $\mathcal{B}.$ Thus, we can choose $h_1,h_2>0$ such that the solution of (\ref{sy}) is non-degenerate.
\end{proof}

\medskip

\begin{remark}\label{rem:h}
Recall that
$$
l(p)=\Delta\log h_1(p)-\rho_2+8\pi-2K(p)
$$
plays a crucial role in the arguments. A consequence of Theorem \ref{th2.1.2} is that we can always choose $h_1$ and $h_2$ to make both the solutions to the shadow system \eqref{sy} non-degenerate and $l(p)\neq 0$.
\end{remark}
\vspace{1cm}

\section{The set of blowing up solutions} \label{sec:blowup}
The aim of this section is to describe the set of all possible bubbling solutions of (\ref{s}): in particular, we shall prove that they are contained in the set $S_{\rho_1}(p,w)\times S_{\rho_2}(p,w)$ when $\rho_1\rightarrow8\pi,$ $\rho_2\notin8\pi\mathbb{N}$, where the definition of $S_{\rho_i}(p,w),i=1,2$ is given in (\ref{4.14}) and
(\ref{4.15}). The latter description will be  used to calculate the topological degree of \eqref{s}. For the sake of simplicity we assume $M$ has a flat metric near a neighborhood of each blow up point (for the general description, see for example \cite{cl1}).

\medskip

The strategy is the following: observing that the first equation in (\ref{s}) can be written as
$$
\Delta v_{1k}+\rho_{1k}\left(\frac{\tilde{h}_ke^{v_{1k}}}{\int_{M}\tilde{h}_ke^{v_{1k}}}-1\right)=0,
$$
where
$$
\tilde{h}_k=h_1e^{-v_{2k}}.
$$
Since $\tilde{h}_k\rightarrow h$ in $C^{2,\alpha}(M)$, see Section \ref{sec:shadow}, all the estimates in \cite{chenlin, cl1} can be applied in this framework. To this end we recall now all the tools introduced in \cite{chenlin, cl1}, which are now based on a non-degenerate solution $(p,w)$ of (\ref{sy}).

\medskip

Let $(p,w)$ be a non-degenerate solution of (\ref{sy}) and set
\begin{equation}
\label{4.1}
h=h_1e^{-w}.
\end{equation}
We notice that
\begin{equation}
\label{4.2}
\nabla_x\big(\log h+4\pi R(x,x)\big)\mid_{x=p}=\nabla_x\big(\log h(x)+8\pi R(x,p)\big)\mid_{x=p}=0.
\end{equation}
For a point $q$ such that $|q-p|\ll1$ and $\lambda\gg 1,$ we introduce
\begin{align}
\label{4.3}
U(x)=\log\left( \frac{e^\lambda}{ \left(1+\frac{\rho_1h(q)}{8}\,e^{\lambda}|x-q|^{2}\right)^2}  \right).
\end{align}
It is known that $U(x)$ satisfies
\begin{align}
\label{4.4}
\Delta U(x)+\rho_1h(q)\,e^{U}=0 \quad \mathrm{in}~\mathbb{R}^2,\qquad U(q)=\max_{\mathbb{R}^2}U(x)=\lambda.
\end{align}
Following the argument in \cite{chenlin, cl1} we define
\begin{align}
\label{4.5}
H(x)=\exp\left\{\log\frac{h(x)}{h(q)}+8\pi R(x,q)-8\pi R(q,q)\right\}-1,
\end{align}
and
\begin{equation}
\label{4.6}
s=\lambda+2\log\left(\frac{\rho_1h(q)}{8}\right)+8\pi R(q,q)+2\frac{\Delta H(q)}{\rho_1h(q)}\frac{\lambda^2}{e^{\lambda}}\;.
\end{equation}

\begin{remark}
We point out that in case we do not have flat metric around the blow up points, the function $H$ in \eqref{4.5} should be modified using a conformal function $\phi$ which in particular it satisfies $\D\phi= -2Ke^\phi$. Keeping this in mind, in the following arguments one gets $\D H(q)=l(q)\,C$.
\end{remark}

\medskip

Furthermore, let $\sigma_0(t)$ be a cut-off function:
\begin{equation} \label{sigma}
\sigma_0(t)=\left\{
\begin{array}{ll}
1~&\mathrm{if}~|t|<r_0,\\
0~&\mathrm{if}~|t|\geq 2r_0.
\end{array}\right.
\end{equation}
Set $\sigma(x)=\sigma_0(|x-q|)$ and
\begin{align*}
J(x)=\left\{\begin{array}{ll}
\Big(H(x)-\nabla H(q)\cdot(x-q)\Big)\sigma(x),~&x\in B_{2r_0}(q),\\
0,&x\notin B_{2r_0}(q).
\end{array}\right.
\end{align*}
Finally, let $\eta(x)$ such that
\begin{equation}
\label{4.7}
\left\{\begin{array}{l}
\Delta\eta+\rho_1 h(q)\,e^{U}\bigr(\eta+J(x)\bigr)=0\qquad\mathrm{on}~\mathbb{R}^2,\\
\eta(q)=0,\quad\nabla\eta(q)=0.
\end{array}\right.
\end{equation}
The existence of the above function $\eta$ was proved in \cite{cl1}. Furthermore, the following result holds true.
\begin{lemma} (\cite{cl1})
\label{le4.1}
Let $R=\sqrt{\frac{\rho_1h(q)}{8}\,e^{\lambda}}.$ For $h\in C^{2,\alpha}(M)$ and large $\lambda$ there exists a solution $\eta$ satisfying (\ref{4.7}) and
the following:
\begin{itemize}
  \item [(i)] $\displaystyle{\eta(x)=-\frac{8\Delta H(q)}{\rho_1h(q)}e^{-\lambda}\Big(\log\big(R|x-q|+2\big)\Big)^2+O(\lambda e^{-\lambda}) \qquad \mathrm{on}~B_{2r_0}(q)}$, \vspace{0.2cm}
  \item [(ii)] $\eta,\,\nabla_x\eta,\,\partial_q\eta,\,\partial_{\lambda}\eta,\,\nabla_x\partial_q\eta,\,\nabla_x\partial_{\lambda}\eta=
              O(\lambda^2e^{-\lambda}) \qquad \mathrm{on}~B_{2r_0}(q)$.
\end{itemize}
\end{lemma}

\medskip

\noindent The blowing up solutions will be very well approximated by the following functions $v_{q,\lambda,a}$, see Proposition~\ref{le4.2} below:
\begin{align}
\label{4.8}
\left\{\begin{array}{ll}
v_q(x)=\Big(U(x)+\eta(x)+8\pi\big(R(x,q)-R(q,q)\big)+s\Big)\sigma(x)+8\pi G(x,q)\big(1-\sigma(x)\big),\vspace{0.2cm}\\
\overline{v}_q=\frac{1}{|M|}\int_Mv_q,\vspace{0.2cm}\\
v_{q,\lambda,a}=a(v_q-\overline{v}_q).
\end{array}\right.
\end{align}
We notice that $v_q(x)$ depends on $q$ and $\lambda.$ We will show that the error term in the approximation belongs to following sets:
\begin{align}
\label{4.9}
O^{(1)}_{q,\lambda}=&\left\{\phi\in\mathring{H}^1(M)~\Big{|}~\int_M\nabla\phi\cdot\nabla v_{q}=\int_M\nabla\phi\cdot\nabla\partial_{q}v_{q}=
\int_M\nabla\phi\cdot\nabla\partial_{\lambda}v_{q}=0\right\},
\end{align}
and
\begin{equation}
\label{4.10}
O^{(2)}_{q,\lambda}=\left\{\psi\in W^{2,\mathfrak{p}}(M)~\Big{|}~\int\psi=0\right\},\quad \mathfrak{p}>2.
\end{equation}
The idea will be then to consider the following decomposition
$$\mathring{H}^1=O_{q,\lambda}^{(1)}\bigoplus~\biggr\{\mathrm{linear~subspace~spanned~by}~v_{q},
\partial_{\lambda}v_{q}~\mathrm{and}~\partial_{q}v_{q}\biggr\}.$$
For future references, for any $(q,\lambda)$ we define
\begin{equation}
\label{4.11}
t=\lambda+8\pi R(q,q)+2\log \frac{\rho_1h(q)}{8}+\frac{2\Delta H(q)}{\rho_1 h(q)}\lambda^2e^{-\lambda}-\overline{v}_q = s - \overline{v}_q.
\end{equation}

\medskip

The last construction is concerned with $\rho_1\neq 8\pi$. For a non-degenerate solution $(p,w)$ of (\ref{sy}) we define $\lambda(\rho_1)$ such that
\begin{align}
\label{4.12}
\rho_1-8\pi=\frac{2\big(\Delta\log h(p)+8\pi-2K(p)\big)}{h(p)}\lambda(\rho_1)\,e^{-\lambda(\rho_1)},
\end{align}
where $K(p)$ is the Gaussian curvature in $p$. By (\ref{sy}) we have $e^{-8\pi G(x,p)}\mid_{x=p}=0$ and $\Delta w(p)=\rho_2.$ Therefore
\begin{align}
\label{4.13}
\Delta\log h(p)+8\pi-2K(p)=&\Delta\log h_1(p)-\rho_2+8\pi-2K(p).
\end{align}
We stress that to  be $\lambda(\rho_1)$ well-defined one has to require
$$\Delta\log h_1(p)-\rho_2+8\pi-2K(p)\neq0.$$

\

Let $c_1$ be a positive constant, which will be chosen later. Recall the definitions of $O_{q,\lambda}^{(1)}$, $O_{q,\lambda}^{(2)}$ in \eqref{4.9} and \eqref{4.10}, respectively. For $\rho_1\neq 8\pi$ we set
\begin{align}
\label{4.14}
S_{\rho_1}(p,w)=\Big\{&v_1=v_{q,\lambda,a}+\phi~\Big{|}~|q-p|\leq c_1\lambda(\rho_1)\,e^{-\lambda(\rho_1)},\nonumber\\
&|\lambda-\lambda(\rho_1)|\leq c_1\lambda(\rho_1)^{-1},~
|a-1|\leq c_1\lambda(\rho_1)^{-\frac12}e^{-\lambda(\rho_1)},\nonumber\\
&\phi\in O_{q,\lambda}^{(1)}
~\mathrm{and}~\|\phi\|_{H^1(M)}\leq c_1\lambda(\rho_1)\,e^{-\lambda(\rho_1)}\Big\},
\end{align}
and
\begin{equation}
\label{4.15}
S_{\rho_2}(p,w)=\Big\{v_2=w+\psi~\Big{|}~\psi\in O_{q,\lambda}^{(2)}~\mathrm{and}~\|\psi\|_*\leq c_1\lambda(\rho_1)\,e^{-\lambda(\rho_1)}\Big\},
\end{equation}
where $\|\psi\|_*=\|\psi\|_{W^{2,\mathfrak{p}}(M)}.$\\

The goal is to prove that for a sequence of bubbling solutions $(v_{1k},v_{2k})$ of (\ref{s}), $\rho_{1k}\to8\pi$, $\rho_2\notin8\pi\N$, we have
$$(v_{1k},v_{2k})\in S_{\rho_{1k}}(p,w)\times S_{\rho_2}(p,w).$$
More precisely:
\begin{proposition}
\label{le4.2}
Let $(v_{1k},v_{2k})$ be a sequence of blow up solutions of (\ref{s}) for $\rho_{1k}\to8\pi$, $\rho_2\notin8\pi\N$: in particular $v_{1k}$ blows up at $p$, weakly converges to $8\pi G(x,p)$ and $v_{2k}\rightarrow w$ in $C^{2,\alpha}(M).$ Suppose $(p,w)$ is a non-degenerate solution of (\ref{sy}) and
\begin{align}
\label{4.16}
\Delta\log h_1(p)-\rho_2+8\pi-2K(p)\neq 0.
\end{align}
Then, there exist $q^*_k$, $\lambda^*_k,$ $a^*_k,$ $\phi^*_k,$ $\psi^*_k$ such that
\begin{equation}
\label{4.17}
v_{1k}=v_{q^*_k,\lambda^*_k,a^*_k}+\phi^*_k,~\quad v_{2k}=w+\psi^*_k,
\end{equation}
and
$(v_{1k},v_{2k})\in S_{\rho_{1k}}(p,w)\times S_{\rho_2}(p,w)$.
\end{proposition}

\begin{proof} Recall that $v_{1k}$ and $v_{2k}$ satisfy
\begin{equation}
\label{4.18}
\left\{\begin{array}{l}
\Delta v_{1k}+\rho_{1k}\Big(\frac{h_1e^{v_{1k}-v_{2k}}}{\int_{M}h_1e^{v_{1k}-v_{2k}}}-1\Big)=0,\\
\Delta v_{2k}+\rho_{2}\Big(\frac{h_2e^{v_{2k}-v_{1k}}}{\int_{M}h_2e^{v_{2k}-v_{1k}}}-1\Big)=0.
\end{array}\right.
\end{equation}
We write the first equation of the above system as
\begin{equation}
\label{4.19}
\Delta v_{1k}+\rho_{1k}\left(\frac{\tilde{h}_ke^{v_{1k}}}{\int_{M}\tilde{h}_ke^{v_{1k}}}-1\right)=0,
\end{equation}
where
\begin{equation}
\label{4.20}
\tilde{h}_k=h_1e^{-v_{2k}}=he^{-\psi_k} \qquad \mathrm{and} \qquad \psi_k=v_{2k}-w.
\end{equation}
We recall now the following fact: since $\tilde{h}_k\rightarrow h$ in $C^{2,\alpha}(M)$, see Section \ref{sec:shadow}, all the estimates in \cite{chenlin, cl1} can be applied in this framework. This will lead to the approximation of $v_{1k}$. In the second step we use the latter approximation jointly with the non-degeneracy of the shadow system \eqref{sy} to get the estimate of the error term in the approximation of $v_{2k}$ and in turn of the error of $v_{1k}$.

\

We follow the arguments in \cite{chenlin, cl1}. All the details can be found in these papers. Let $\tilde{q}_k$ be the maximal point of $\tilde{v}_{1k}$ near $p$, where
$$
\tilde{v}_{1k}=v_{1k}-\log\int_M\tilde{h}_k\,e^{v_{1k}}.
$$
As in page 13 of \cite{chenlin} we let
\begin{align*}
\lambda_{k}=\tilde{v}_{1k}(\tilde{q}_{k})-\log\int_M\tilde{h}_k\,e^{v_{1k}}.
\end{align*}
Around $\tilde{q}_{k}$ we set
\begin{equation*}
\tilde{U}_{k}(x)=\log\frac{e^{\lambda_k}}{\left(1+\frac{\rho_{1k}\tilde{h}_k(q_{k})}{8}\,e^{\lambda_{k}}|x-q_{k}|^2\right)^2}\;.
\end{equation*}
where $q_k$ is chosen such that
$$\nabla\tilde{U}_k(\tilde{q}_k)=\nabla\log \tilde{h}_k(\tilde{q}_k).$$
It is easy to check that $|q_k-\tilde{q}_k|=O(e^{-\lambda_k}).$
Then the error terms of the approximation inside and outside $B_{r_0}(q_k)$ are given by
\begin{equation}
\label{4.21}
\tilde{\eta}_{k}(x)=\tilde{v}_{1k}-\tilde{U}_{k}(y)-\bigr(8\pi R(x,q_k)-8\pi R(q_k,q_{k})\bigr), \qquad \mbox{in } B_{r_0}(q_k),
\end{equation}
\begin{equation}
\label{4.22}
\xi_k(x)=v_{1k}(x)-8\pi G(x,q_{k})= \tilde v_{1k}(x)-8\pi G(x,q_{k})-\bar{\tilde {v}}_{1k}, \qquad \mbox{in } M\setminus B_{r_0}(q_k).
\end{equation}
By the Lemma 5.3 in \cite{chenlin} we have
\begin{equation}
\label{4.23}
\xi_k(x)=O(\lambda_ke^{-\lambda_k}) \qquad \mathrm{for}~x\in M\setminus B_{r_0}(q_k).
\end{equation}
By a straightforward computation, see page 20 in \cite{chenlin}, the error term $\tilde{\eta}_k$ satisfies
\begin{equation}
\label{4.24}
\Delta\tilde{\eta}_{k}+\rho_{1k}\,\tilde{h}_k(q_{k})\,e^{\tilde{U}_{k}}\tilde{H}_{k}(x,\tilde{\eta}_{k})=0,
\end{equation}
where (see also \eqref{4.28})
\begin{align*}
\tilde{H}_{k}(x,t)=&\exp\left\{\log\frac{\tilde{h}_k(x)}{\tilde{h}_k(q_{k})}+8\pi\big(R(x,q_k)-R(q_k,q_{k})\big)+t\right\}-1\\
=&H_{k}(x)+t+O(|t|^2),
\end{align*}
and
\begin{align*}
H_k(x)=\exp\left\{\log \frac{\tilde{h}_k(x)}{\tilde{h}_k(q_k)}+8\pi R(x,q_k)-8\pi R(q_k,q_k)\right\}-1.
\end{align*}
Except for the higher-order terms, equation (\ref{4.24}) resembles the one in (\ref{4.7}). By Theorem 1.4 in \cite{chenlin} one has
\begin{equation}
\label{4.25}
\tilde{\eta}_{k}(x)=-\frac{8}{\rho_{1k}\,\tilde{h}_k(q_{k})}\Delta H_k(q_{k})\,e^{-\lambda_{k}}\bigr(\log(R_{k}|x-q_{k}|+2)\bigr)^2
+O(\lambda_{k}\,e^{-\lambda_{k}}),
\end{equation}
for $x\in B_{2r_0}(q_{k}),$ where $R_k=\sqrt{\frac{\rho_{1k}\tilde{h}_k(q_{k})}{8}e^{\lambda_{k}}}.$

\medskip

Moreover, from \cite[Theorem 1.1, Theorem 1.4 and Lemma 5.4]{chenlin}, we have the following estimates:
\begin{equation}
\label{4.26}
\rho_{1k}-8\pi=\frac{2\big(\Delta\log\tilde{h}_k(q_k)+8\pi-2K(q_k)\big)}{\tilde{h}_k(q_{k})}\lambda_{k}e^{-\lambda_{k}}
+O(e^{-\lambda_{k}}),
\end{equation}
\begin{align}
\label{4.27}
\overline{\tilde{v}}_{1k}+\lambda_{k}+2\log\frac{\rho_{1k}\tilde{h}_k(q_{k})}{8}+8\pi R(q_k,q_{k})+
\frac{2\Delta H_k(q_{k})}{\rho_{1k}\tilde{h}_k(q_{k})}\lambda_k^2\,e^{-\lambda_{k}}=O(\lambda_{k}\,e^{-\lambda_{k}}),
\end{align}
and
\begin{equation}
\label{4.28}
|\nabla H_{k}(q_{k})|=O(\lambda_{k}\,e^{-\lambda_{k}}).
\end{equation}

\medskip

With these preparations, we now let $\eta_{k}$ be defined as in (\ref{4.7}) and $v_{q_k,\lambda_k,a_k}$ be defined as in (\ref{4.8}) with $q=q_k$, $\lambda=\lambda_k$, $a=a_k=1$ and $h$ replaced by $h_k$. We start by observing that from Lemma \ref{le4.1} and (\ref{4.25}) we deduce
\begin{align}
\label{4.29}
\eta_{k}(x)=\tilde{\eta}_{k}+O(\lambda_{k}e^{-\lambda_{k}}) \qquad \mathrm{for}~x\in B_{2r_0}(q_k).
\end{align}
For $x\in B_{r_0}(q_k),$ recalling the definition of $s$ in \eqref{4.6}:
\begin{align*}
v_{q_k,\lambda_k,a_k}=~&\tilde{U}_{k}(x)+\eta_{k}(x)+\big(8\pi R(x,q_k)-8\pi R(q_k,q_{k})\big)+\lambda_{k}+2\log\frac{\rho_{1k}\tilde{h}_k(q_{k})}{8}\\
&+8\pi R(q_k,q_{k})+\frac{2\Delta H_k(q_{k})}{\rho_{1k}\tilde{h}_k(q_{k})}\lambda_{k}^2e^{-\lambda_{k}}-\overline{v}_{q_k},
\end{align*}
where $\overline{v}_{q_k}$ stands for the average of $v_{q_k}$. From \cite[Lemma 2.2 and Lemma 2.3]{cl1}, we have
\begin{align}
\label{4.30}
\begin{split}
&v_{q_k}-8\pi G(x,q_k)=O(\lambda_k\,e^{-\lambda_k}) \qquad \mathrm{in}~M\setminus B_{r_0}(q_k), \\
&\overline{v}_{q_k}=O(\lambda_k\,e^{-\lambda_k}).
\end{split}
\end{align}
By (\ref{4.21}), (\ref{4.27}), (\ref{4.29}) and (\ref{4.30}), we have
\begin{align}
\label{4.31}
v_{1k}-v_{q_k,\lambda_k,a_k}=~&\tilde{v}_{1k}+\log \int_{M}\tilde{h}_ke^{v_{1k}}-v_{q_k,\lambda_k,a_k}\nonumber\\
=~&\tilde{v}_{1k}-\tilde{U}_{k}-\big(8\pi R(x,x)-8\pi R(x,q_{k})\big)-\eta_{k}(x)+O(\lambda_{k}\,e^{-\lambda_{k}})\nonumber\\
=~&\tilde{\eta}_{k}(x)-\eta_{k}(x)+O(\lambda_{k}\,e^{-\lambda_{k}})=O(\lambda_{k}\,e^{-\lambda_{k}})
\end{align}
for $x\in B_{r_0}(q_k)$. For $x\in M\setminus B_{r_0}(q_k),$ by (\ref{4.22}) and (\ref{4.30}) we get
\begin{align*}
v_{1k}-v_{q_k,\lambda_k,a_k}=v_{1k}-8\pi G(x,q_k)-(v_{q_k}-8\pi G(x,q_k))+\overline{v}_{q_k}=O(\lambda_k\,e^{-\lambda_k}).
\end{align*}
Thus, we conclude that
\begin{align}
\label{4.32}
v_{1k}=v_{q_k,\lambda_k,a_k}+\phi_{k}, \qquad \mathrm{where}~\|\phi_{k}\|_{L^{\infty}(M)}<\tilde{c}\lambda_k\,e^{-\lambda_k},
\end{align}
where $\tilde{c}$ is independent of $\psi_k.$

\

Next, to get the estimate of the error $\psi_k$ we \emph{evaluate} it on the linearized operator of the first equation in the shadow system \eqref{sy}. By using $v_{2k}=w+\psi_k$ and the second equation of (\ref{4.18}) we get
\begin{align}
\label{4.33}
\mathcal{L}(\psi_{k})=\mathbb{I}_1+\mathbb{I}_2+\mathbb{I}_3,
\end{align}
where
\begin{align*}
\mathcal L(\psi_k)=~&\Delta\psi_{k}+\rho_{2}\frac{h_2e^{w-8\pi G(x,p)}}{\int_{M}h_2e^{w-8\pi G(x,p)}}\psi_{k}\\
&-\rho_{2}\frac{h_2e^{w-8\pi G(x,p)}}
{(\int_Mh_2e^{w-8\pi G(x,p)})^2}\int_{M}(h_2e^{w-8\pi G(x,p)}\psi_{k})\\
&-8\pi\rho_{2}\frac{h_2e^{w-8\pi G(x,p)}}{\int_{M}h_2e^{w-8\pi G(x,p)}}\big(\nabla G(x,p)(q_{k}-p)\big)\\
&+8\pi\rho_{2}\frac{h_2e^{w-8\pi G(x,p)}}{(\int_{M}h_2e^{w-8\pi G(x,p)})^2}
\int_{M}\big(h_2e^{w-8\pi G(x,p)}(\nabla G(x,p)(q_{k}-p))\big),
\end{align*}
\begin{align*}
\mathbb{I}_1=&-\rho_{2}\frac{h_2e^{w+\psi_{k}-v_{1k}}}{\int_{M}h_2e^{w+\psi_{k}-v_{1k}}}
+\rho_{2}\frac{h_2e^{w+\psi_{k}-8\pi G(x,q_k)}}
{\int_{M}h_2e^{w+\psi_{k}-8\pi G(x,q_k)}},
\end{align*}
\begin{align*}
\mathbb{I}_2=&\rho_{2}\frac{h_2e^{w-8\pi G(x,p)}}{\int_{M}h_2e^{w-8\pi G(x,p)}}
-\rho_{2}\frac{h_2e^{w+\psi_{k}-8\pi G(x,p)}}{\int_{M}h_2e^{w+\psi_{k}-8\pi G(x,p)}}\\
&+\rho_{2}\frac{h_2e^{w-8\pi G(x,p)}}{\int_{M}h_2e^{w-8\pi G(x,p)}}\psi_{k}\\
&-\rho_{2}\frac{h_2e^{w-8\pi G(x,p)}}
{(\int_{M}h_2e^{w-8\pi G(x,p)})^2}\int_{M}(h_2e^{w-8\pi G(x,p)}\psi_{k}),
\end{align*}
and
\begin{align*}
\mathbb{I}_3=&-\rho_{2}\frac{h_2e^{w+\psi_{k}-8\pi G(x,q_k)}}{\int_{M}h_2e^{w+\psi_{k}-8\pi G(x,q_k)}}
+\rho_{2}\frac{h_2e^{w+\psi_{k}-8\pi G(x,p)}}{\int_{M}h_2e^{w+\psi_{k}-8\pi G(x,p)}}\\
&-8\pi\rho_{2}\frac{h_2e^{w-8\pi G(x,p)}}{\int_{M}h_2e^{w-8\pi G(x,p)}}(\nabla G(x,p)(q_{k}-p))\\
&+8\pi\rho_{2}\frac{h_2e^{w-8\pi G(x,p)}}{(\int_{M}h_2e^{w-8\pi G(x,p)})^2}
\int_{M}\big(h_2e^{w-8\pi G(x,p)}(\nabla G(x,p)(q_{k}-p))\big).
\end{align*}
Reasoning as in \cite{lwy} and decomposing the domain into $B_{r_0}(q_{k})$ and $M\setminus B_{r_0}(q_{k})$, it is not difficult to show that  $\mathbb{I}_1=O(\lambda_{k}\,e^{-\lambda_{k}})$ and $\mathbb{I}_2=O(\|\psi_k\|^2_{*}).$ Concerning $\mathbb{I}_3$, we divide it into three parts: $\mathbb{I}_3=\mathbb{I}_{31}+\mathbb{I}_{32}+\mathbb{I}_{33},$
where
\begin{align*}
\mathbb{I}_{31}=~&-\rho_{2}\frac{h_2e^{w+\psi_{k}-8\pi G(x,q_{k})}}
{\int_{M}h_2e^{w+\psi_{k}-8\pi G(x,q_{k})}}+\rho_{2}\frac{h_2e^{w+\psi_{k}-8\pi G(x,p)}}
{\int_{M}h_2e^{w+\psi_{k}-8\pi G(x,p)}}\\
&-8\pi\rho_{2}\frac{h_2e^{w+\psi_k-8\pi G(x,p)}}{\int_{M}h_2e^{w+\psi_k-8\pi G(x,p)}}\big(\nabla G(x,p)(q_{k}-p)\big)\\
&+8\pi\rho_{2}\frac{h_2e^{w+\psi_k-8\pi G(x,p)}}{(\int_{M}h_2e^{w+\psi_k-8\pi G(x,p)})^2}
\int_{M}\Big(h_2e^{w+\psi_k-8\pi G(x,p)}(\nabla G(x,p)(q_{k}-p))\Big),
\end{align*}
\begin{align*}
\mathbb{I}_{32}=~&8\pi\rho_{2}\frac{h_2e^{w-8\pi G(x,p)}}{(\int_{M}h_2e^{w-8\pi G(x,p)})^2}
\int_{M}\Big(h_2e^{w-8\pi G(x,p)}(\nabla G(x,p)(q_{k}-p))\Big)\\
&-8\pi\rho_{2}\frac{h_2e^{w+\psi_k-8\pi G(x,p)}}{(\int_{M}h_2e^{w+\psi_k-8\pi G(x,p)})^2}
\int_{M}\Big(h_2e^{w+\psi_k-8\pi G(x,p)}(\nabla G(x,p)(q_{k}-p))\Big),
\end{align*}
and
\begin{align*}
\mathbb{I}_{33}=~&8\pi\rho_{2}\frac{h_2e^{w+\psi_k-8\pi G(x,p)}}{\int_{M}h_2e^{w+\psi_k-8\pi G(x,p)}}\big(\nabla G(x,p)(q_{k}-p)\big)\\
&-8\pi\rho_{2}\frac{h_2e^{w-8\pi G(x,p)}}{\int_{M}h_2e^{w-8\pi G(x,p)}}\big(\nabla G(x,p)(q_{k}-p)\big).
\end{align*}
It is not difficult to see
$$\mathbb{I}_{31}=O(|q_k-p|^2), \quad \mathbb{I}_{32}=O(1)\|\psi\|_*|q_k-p|, \quad \mathbb{I}_{33}=O(1)\|\psi\|_*|q_k-p|.$$
In conclusion we have the estimate
\begin{align}
\label{4.34}
\mathcal{L}(\psi_k)=o(1)\|\psi_k\|_*+O\bigr(\|\psi_k\|^2_{*}+\lambda_ke^{-\lambda_k}\bigr)+O(|p-q_k|^2).
\end{align}

\

Now, to get the estimate of the error $|p-q_k|$ we \emph{evaluate} it on the linearized operator of the second equation in the shadow system \eqref{sy}. By the definition of $H_k$ (see below \eqref{4.24}) and (\ref{4.28}), we have
\begin{equation}
\label{4.35}
\nabla H_k(q_k)=\nabla\log h(q_k)-\nabla\psi_k(q_k)+8\pi\nabla R(q_k,q_k)=O(\lambda_ke^{-\lambda_k}).
\end{equation}
By (\ref{4.2}) and (\ref{4.35}), using Taylor's expansion we have
\begin{align}
\label{4.36}
&\nabla^2\big{(}\log h(p)+8\pi R(p,p)\big{)}(q_{k}-p)-\nabla\psi_k(p)\nonumber\\
=~&\nabla\log h(q_k)-\nabla\psi_k(q_k)+8\pi\nabla R(q_k,q_k)-\big(\nabla\log h(p)+8\pi\nabla R(p,p)\big)\nonumber\\
&+\nabla\psi_k(q_k)-\nabla\psi_k(p)+O(|p-q_k|^2)\nonumber\\
=~&\nabla H_k(q_k)-\nabla H(p)+O\bigr(|p-q_k|^{\gamma}\|\psi_k\|_*\bigr)+O(|p-q_k|^2),\\
=~&O(\lambda_ke^{-\lambda_k})+O\bigr(|p-q_k|^{\gamma}\|\psi_k\|_*\bigr)+O(|p-q_k|^2),
\end{align}
where $\gamma$ depends on $\mathfrak{p}$. In the last step we used $\nabla H(p)=0$.

\medskip

Using the estimates (\ref{4.34})-(\ref{4.36}) and the non-degeneracy of $(p,w)$, we obtain
\begin{equation}
\label{4.37}
\|\psi_k\|_*+|p-q_k|\leq C\biggr(\lambda_ke^{-\lambda_k}+o(1)\|\psi_k\|_*+\|\psi_k\|^2_{*}+|p-q_k|^2\biggr),
\end{equation}
where $C$ is a constant independent of $k$ and $\psi_k.$ Therefore, we have
\begin{equation}
\label{4.38}
\|\psi_k\|_*=O(\lambda_ke^{-\lambda_k}), \quad |p-q_k|=O(\lambda_ke^{-\lambda_k}).
\end{equation}
By the above estimates, recalling now (\ref{4.12}), (\ref{4.20}) and (\ref{4.26}) we deduce
\begin{equation}
\label{4.39}
\lambda_k-\lambda(\rho_{1k})=O(\lambda(\rho_{1k})^{-1}), \quad \tilde{h}_k=h+O(\lambda(\rho_{1k})\,e^{-\lambda(\rho_{1k})}),
\quad |q_{k}-p|=O(\lambda(\rho_{1k})\,e^{-\lambda(\rho_{1k})})
\end{equation}
and
\begin{equation}
\label{4.40}
v_{2k}-w=O(\lambda(\rho_{1k})\,e^{-\lambda(\rho_{1k})}).
\end{equation}
We replace $\tilde{h}_k$ by $h$ in the definition of $v_{q}$ and we denote the new function still by $v_{q}.$
By the second estimate in (\ref{4.38}) we have
\begin{align*}
v_{q_k}-v_{q}=O(\lambda(\rho_{1k})\,e^{-\lambda(\rho_{1k})}).
\end{align*}
We set
\begin{align}
\label{4.41}
v_{q,\lambda,a}=v_{q}-\overline{v}_{q}.
\end{align}
By (\ref{4.32}) and (\ref{4.41}) it follows
\begin{align}
\label{4.42}
v_{1k}-v_{q,\lambda,a}=O(\lambda(\rho_{1k})\,e^{-\lambda(\rho_{1k})}).
\end{align}
Finally, once we get the existence of $v_{q,\lambda,a}$ with the above property, by \cite[Lemma 3.2]{cl1} one can deduce that there exists a unique triplet $(q^*_k,\lambda^*_k,a^*_k)$ that satisfy the condition in the definition \eqref{4.14} of $S_{\rho_{1k}}(p,w)$ and
$\phi^*\in O^{(1)}_{q^*_k,\lambda^*_k}$ such that
\begin{equation}
\label{4.43}
v_{1k}=v_{q^*_k,\lambda^*_k,a^*_k}+\phi_k^*\,.
\end{equation}
Therefore, we conclude
$$(v_{1k},v_{2k})\in S_{\rho_{1k}}(p,w)\times S_{\rho_2}(p,w).$$
\end{proof}

\medskip

\noindent From the latter result we are able to characterize the blow up situation.
\begin{theorem}
\label{th4.1}
Suppose $h_1,h_2$ are two positive $C^{2,\alpha}$ function on $M$ such that
\begin{itemize}
\item[(a)]Any solution $(p,w)$ of (\ref{sy}) is non-degenerate.
\item[(b)] $\Delta\log h_1(p)-\rho_2+8\pi-2K(p)\neq0.$
\end{itemize}
Then, there exist $\varepsilon_0>0$ and $C>0$ such that
for any solution of (\ref{s}) with $\rho_1\in(8\pi-\varepsilon_0,8\pi+\varepsilon_0),\rho_2\notin8\pi\mathbb{N}$, we have the following alternative: either
\begin{itemize}
	\item[(i)] $|v_1|,|v_2|\leq C, \quad\forall x\in M$,
	\item[] or
	\item[(ii)] $(v_1,v_2)\in S_{\rho_1}(p,w)\times S_{\rho_2}(p,w)$ for some solution $(p,w)$ of (\ref{sy}).
\end{itemize}	
\end{theorem}

\medskip

\begin{remark} \label{rem:segno}
As we have pointed out in the Introduction, for a sequence of solutions bubbling around a point $p$, the rate of $|\rho_{1k}-8\pi|$ is related to
$$
l(p)=\Delta\log h_1(p)-\rho_2+8\pi-2K(p).
$$
More precisely, as observed in \cite{cl1} (see also \cite{cl3, lwy}) we get that
$$
	\mbox{sgn}(\rho_{1k}-8\pi) = \mbox{sgn}(l(p)),
$$
see for example \eqref{4.26} (and \eqref{4.39}). We will see how this fact plays a role in the proof of Theorem \ref{th1.6} (see also Theorem \ref{th1.4}).
\end{remark}

\vspace{1cm}

\section{Analysis of the nonlinear operator} \label{sec:operator}
Our final goal is to compute the degree of the following nonlinear operator:
$$\left(\begin{array}{l}v_1\\v_2\end{array}\right)=
(-\Delta)^{-1}\left(\begin{array}{l}
\rho_1\left(\frac{h_1e^{v_1-v_2}}{\int_Mh_1e^{v_1-v_2}}-1\right) \vspace{0.1cm}\\
\rho_2\left(\frac{h_2e^{v_2-v_1}}{\int_Mh_2e^{v_2-v_1}}-1\right)
\end{array}\right).$$
In this section we will analyze the latter operator in the space $S_{\rho_1}(p,w)\times S_{\rho_2}(p,w)$. Set
\begin{align} \label{op}
T(v_1,v_2)=\left(\begin{array}{l}T_1(v_1,v_2)\\T_2(v_1,v_2)\end{array}\right)
=\Delta^{-1}\left(\begin{array}{l}
\rho_1\left(\frac{h_1e^{v_1-v_2}}{\int_Mh_1e^{v_1-v_2}}-1\right) \vspace{0.1cm}\\
\rho_2\left(\frac{h_2e^{v_2-v_1}}{\int_Mh_2e^{v_2-v_1}}-1\right)
\end{array}\right).
\end{align}
We know that the degree can be calculated by considering the contributions of the blowing up solutions as $\rho_1\to8\pi$. We have proved in Proposition \ref{le4.2} that all the blowing up solutions are contained in the set $S_{\rho_1}(p,w)\times S_{\rho_2}(p,w)$, see the definitions \eqref{4.14} and \eqref{4.15}. Furthermore, we will actually prove that such bubbling solutions do exist. The latter result is a byproduct of the degree formula of the operator \eqref{op} (i.e. Theorem \ref{th1.4}), see also the discussion in the Introduction.

\medskip

We shall study the operator \eqref{op} in the space $S_{\rho_1}(p,w)\times S_{\rho_2}(p,w)$. Any $v_1$ in $S_{\rho_1}(p,w)$ is represented by $(q,\lambda,a,\phi)$ while any $v_2$ in $S_{\rho_2}(p,w)$ is represented by $(w,\psi).$ Therefore, the nonlinear operator $v_1+T_1(v_1,v_2)$ can be expressed according to this representation. In this way we will be able to get the leading terms of the latter operator in the set $S_{\rho_1}(p,w)\times S_{\rho_2}(p,w)$. On the other hand, the operator $v_2+T_2(v_1,v_2)$ has a simpler for and it will be studied in the next section. We will see that this will lead to count the degree on a finite-dimensional space (at least for what concerns $v_1+T_1(v_1,v_2)$).

\medskip

We start by analyzing the term $\rho_1h_1e^{v_1-v_2}$. Let $v_1=v_{q,\lambda,a}+\phi\in S_{\rho_1}(p,w)$ and $y=x-q$. Recalling the definitions \eqref{4.3}, \eqref{4.5}, \eqref{4.6}, \eqref{4.8}, \eqref{4.11} and that $t=s-\overline{v}_{q}$,
for $x\in B_{r_0}(q)$ we get
\begin{align*}
v_{q,\lambda,a}(x)+\log\frac{h(x)}{h(q)}=&U+t+H(x)+\eta
+(a-1)(U+s)+O\big(|a-1|(|y|+|\eta|+|\overline{v}_q|)\big).
\end{align*}
It follows that in $B_{r_0}(q)$, by Taylor we have
\begin{align}
\label{5.1}
\rho_1h_1e^{v_1-v_2-\phi+\psi}=\rho_1h(q)\,e^{U+t}\Big[1+(a-1)(U+s)+\eta+H(x)+(a-1)O(|y|)+O(\tilde{\beta}^2)\Big],
\end{align}
where
$$\tilde{\beta}=\lambda|a-1|+|\eta|+|H(x)|+|\overline{v}_q|.$$
Therefore, letting $\varphi=\phi-\psi$ we have in $B_{r_0}(q)$
\begin{align}
\label{5.2}
\rho_1h_1e^{v_1-v_2}=~&(1+\varphi)\rho_1h\,e^{v_{q,\lambda,a}}+(e^{\varphi}-1-\varphi)\rho_1h\,e^{v_{q,\lambda,a}}
\nonumber\\
=~&\rho_1h(q)\,e^{U+t}\Big[1+(a-1)(U+s)+\eta+H(x)+(a-1)O(|y|)+\varphi\Big]+\tilde{E},
\end{align}
where
\begin{equation}
\label{5.3}
\tilde{E}=\bigr(e^{\varphi}-1-\varphi\bigr)\rho_1h\,e^{v_{q,\lambda,a}}+\rho_1h(q)\,e^{U+t}O(\varphi^2+\tilde{\beta}^2).
\end{equation}
Using the latter expression for $\rho_1h_1e^{v_1-v_2}$ we are in position to obtain the following estimate for $\int_M\rho_1h_1e^{v_1-v_2}.$
\medskip

\begin{lemma}
\label{le5.1}
Let $v_1=v_{q,\lambda,a}+\phi\in S_{\rho_1}(p,w)~\mathrm{and}~v_2=w+\psi\in S_{\rho_2}(p,w).$
Then as
$\rho_1\rightarrow 8\pi,\rho_2\notin 8\pi\mathbb{N},$ we have
\begin{align}
\label{5.6}
\int_{M}\rho_1h_1e^{v_1-v_2}=~&8\pi e^{t}(1-\psi(q))
+\frac{16\pi}{\rho_1h(q)}\Delta H(q)\frac{\lambda}{e^{\lambda}}e^t+16\pi\lambda(a-1)e^{t}
+O\bigr(|a-1|e^{\lambda}+1\bigr).
\end{align}
\end{lemma}
\begin{proof}
The proof of Lemma \ref{le5.1} is by direct computations, we refer the readers to \cite[Lemma 4.1]{cl1} and \cite[Lemma 4.3]{lwy} for details.
\end{proof}
\medskip

\noindent By Lemma \ref{le5.1} and noting that $\lambda=t+O(1)$, see \eqref{4.11}, we get
\begin{align}
\label{5.7}
\frac{\int_M\rho_1h_1e^{v_1-v_2}}{e^t}
=~&8\pi-8\pi\psi(q)+\frac{16\pi}{\rho_1h(q)}\Delta H(q)\lambda \,e^{-\lambda}+16\pi\lambda(a-1)\nonumber\\
~&+O(|a-1|)+O(e^{-\lambda}).
\end{align}
It is not difficult to see that $\frac{e^{t}}{\int_Mh_1e^{v_1-v_2}}-1$ should be small. Indeed, we have
\begin{align}
\label{5.8}
\frac{e^{t}}{\int_Mh_1e^{v_1-v_2}}-1=~&\frac{1}{e^{-t}\int_M\rho_1h_1e^{v_1-v_2}}
\left(\rho_1-\frac{\int_M\rho_1h_1e^{v_1-v_2}}{e^{t}}\right)
\nonumber\\
=~&\theta+O(|a-1|)+O(e^{-\lambda}),
\end{align}
where $\theta$ is defined by
\begin{align}
\label{5.9}
\theta=\frac{1}{8\pi}\Big[(\rho_1-8\pi)-\frac{16\pi}{\rho_1h(q)}\Delta H(q)\lambda \,e^{-\lambda}+8\pi\psi(q)
-16\pi\lambda (a-1)\Big].
\end{align}
Let
\begin{align}
\label{5.10}
\beta=\left|\frac{e^{t}}{\int_Mh_1e^{v_1-v_2}}-1\right|+\tilde{\beta},
\end{align}
and
\begin{align}
\label{5.11}
E=(e^{\varphi}-1-\varphi)\frac{\rho_1h_1e^{v_1-v_2}}{\int_Mh_1e^{v_1-v_2}}
+\rho_1h(q)\,e^{U}\big(O(\varphi^2)+O(\beta^2)\big).
\end{align}
Consider now $\frac{e^{t}}{\int_Mh_1e^{v_1-v_2}}=1+\left(\frac{e^{t}}{\int_Mh_1e^{v_1-v_2}}-1\right)$. Then, in $B_{r_0}(q)$ we have by (\ref{5.2})
\begin{align}
\label{5.12}
\frac{\rho_1h_1e^{v_1-v_2}}{\int_Mh_1e^{v_1-v_2}}=~&
(1+\varphi)\frac{\rho_1he^{v_{q,\lambda,a}}}{\int_Mh_1e^{v_1-v_2}}
+(e^{\varphi}-1-\varphi)\frac{\rho_1he^{v_{q,\lambda,a}}}{\int_Mh_1e^{v_1-v_2}}\nonumber\\
=~&\rho_1h(q)\,e^{U}\left[1+\left(\frac{e^{t}}{\int_Mh_1e^{v_1-v_2}}-1\right)
+(a-1)(U+s)\right.\nonumber\\
&+(a-1)O(|y|)+\eta+H+\varphi+O(\beta^2)\biggr]+E.
\end{align}
Thus, in $B_{r_0}(q)$ we get
\begin{align}
\label{5.13}
\Delta\big(v_1+T_1(v_1,v_2)\big)=~&\Delta v_1+\frac{\rho_1h_1e^{v_1-v_2}}{\int_Mh_1e^{v_1-v_2}}-\rho_1\nonumber\\
=~&a(\Delta U+\Delta\eta)+\Delta\phi+8\pi a
+\frac{\rho_1h_1e^{v_1-v_2}}{\int_Mh_1e^{v_1-v_2}}-\rho_1\nonumber\\
=~&\Delta\phi-a\rho_1h(q)\,e^{U}\Big[1+\eta+H-\nabla_yH\cdot y\Big]+8\pi-\rho_1+8\pi(a-1)+\frac{\rho_1h_1e^{v_1-v_2}}{\int_Mh_1e^{v_1-v_2}}\nonumber\\
=~&\Delta\phi+(8\pi-\rho_1)+8\pi(a-1)+\rho_1h(q)\,e^{U}\biggr[(a-1)(U+s-1)+(a-1)O(|y|)+\nonumber\\
&+y\cdot\nabla H+\left(\frac{e^{t}}{\int_Mh_1e^{v_1-v_2}}-1\right)+\varphi\biggr]+E.
\end{align}

\

\noindent Since $v_{q}-\overline{v}_{q}-8\pi G(x,q)$ is small in $B_{2r_0}(q)\setminus B_{r_0}(q)$, see \cite[Lemma 2.2]{cl1}, we just write
\begin{align}
\label{5.14}
\Delta\big(v_1+T_1(v_1,v_2)\big)
=~&\Delta\phi+a\Delta(v_{q}-8\pi G(x,q))+8\pi-\rho_1+8\pi(a-1)\nonumber\\
&+\frac{\rho_1h}{\int_Mh_1e^{v_1-v_2}}
e^{a(v_{q}-\overline{v}_{q}-8\pi G(x,q))+8\pi a G(x,q)+\varphi}.
\end{align}

\noindent On $M\setminus B_{2r_0}(q)$ we have instead
\begin{align}
\label{5.15}
\Delta\big(v_1+T_1(v_1,v_2)\big)=~&\Delta\phi+8\pi-\rho_1+8\pi (a-1)+\frac{\rho_1h}{\int_{M}h_1e^{v_1-v_2}}e^{8\pi a G(x,q)+\varphi-a\overline{v}_q}.
\end{align}

\medskip

From the representation given in (\ref{5.13})-(\ref{5.15}) it is possible to get the leading terms of the operator $v_1+T_1(v_1,v_2)$ in the set $S_{\rho_1}(p,w)\times S_{\rho_2}(p,w)$. Recall that we are considering the decomposition
\begin{equation} \label{dec}
	\mathring{H}^1=O_{q,\lambda}^{(1)}\bigoplus~\biggr\{\mathrm{linear~subspace~spanned~by}~v_{q},
\partial_{\lambda}v_{q}~\mathrm{and}~\partial_{q}v_{q}\biggr\}.
\end{equation}
\begin{proposition}
\label{le5.2}
Let $v_1=v_{q,\lambda,a}+\phi\in S_{\rho_1}(p,w),$ $v_2=w+\psi\in S_{\rho_2}(p,w)$. Then as $\rho_1\rightarrow 8\pi,$ $\rho_2 \notin 8\pi\N$, we have
\begin{enumerate}
      \item
      \begin{align}
      \label{5.16}
      \l\nabla(v_1+T_1(v_1,v_2)),\nabla\phi_1\r=\mathfrak{B}(\phi,\phi_1)
      +O(\lambda \,e^{-\lambda})\|\phi_1\|_{H^1_0(M)},
      \end{align}
      where
      $$\mathfrak{B}(\phi,\phi_1):=\int_{M}\nabla\phi\cdot\nabla\phi_1
      -\int_{B_{r_0}(q)}\rho_1h(q)\,e^{U}\phi\phi_1,$$
      is a positive symmetric, bilinear form satisfying $\mathfrak{B}(\phi,\phi)\geq c_0\|\phi\|_{H^1(M)}^2$ for some constant
      $c_0>0.$
      \item
      \begin{align}
      \label{5.17}
      \l\nabla(v_1+T_1(v_1,v_2)),\nabla\partial_{q}v_{q}\r =& -8\pi\nabla H(q)+8\pi\nabla\psi(q)\nonumber\\
       &+O\left(\lambda|a-1|+\Big|\frac{e^t}{\int_Mh_1e^{v_1-v_2}}-1-\psi(q)\Big|
      +\lambda e^{-\lambda}\right),
      \end{align}
      \item
      \begin{align}
      \label{5.18}
      \l\nabla(v_1+T_1(v_1,v_2)),\nabla\partial_{\lambda}v_{q}\r = &-16\pi(a-1)\Big(\lambda-1+\log\frac{\rho_1h(q)}{8}+4\pi R(q,q)\Big)-8\pi(\theta-\psi(q))\nonumber\\
        &+ O\left(|a-1|+\lambda^2e^{-\frac32\lambda}\right)
      \end{align}
      \item
      \begin{align}
      \label{5.19}
      \l\nabla(v_1+T_1(v_1,v_2)),\nabla v_{q}\r
      =~&\Big(2\lambda-2+8\pi R(q,q)+2\log\frac{\rho_1h(q)}{8}\Big)\nonumber\\
        &\times\l\nabla(v_1+T_1(v_1,v_2)),\nabla\partial_{\lambda}v_{q}\r+16\pi(a-1)\lambda\nonumber\\ &+O(1)\|\phi\|_{H^1(M)}+O(\lambda \,e^{-\lambda}),
      \end{align}
\end{enumerate}
\end{proposition}

\medskip

\noindent We can prove Proposition \ref{le5.2} by straightforward computations following the argument in \cite{cl1, lwy} and the sketch of the process is given in the Appendix.

\medskip

We point out that being $\mathfrak{B}(\phi,\phi_1)$ positive definite will be crucially used in the degree analysis. Roughly speaking, we can deduce that the part concerning $\phi$ does not affect the sign of the total degree.

\vspace{1cm}

\section{Proof of the main Theorems \ref{th1.4}-\ref{th1.6}} \label{sec:proof}
In this section we shall prove the main Theorems \ref{th1.4}-\ref{th1.6} and Theorem \ref{bubb}. We are concerned with the topological degree of the operators $v_i+T_i(v_1,v_2)$, $i=1,2$. The strategy will be the following: we just note that due to the decomposition \eqref{dec}, $v_1=v_{q,\lambda,a}+\phi,~v_2=w+\psi$ is a solution of $v_1+T_1(v_1,v_2)=0,$ if and only if all the left-hand sides of
(\ref{5.16})-(\ref{5.19}) vanish.

In order to solve the system (\ref{5.16})-(\ref{5.19}) and $v_2+T_2(v_1,v_2)=0$, the first step is to deform the operators $v_i+T_i(v_1,v_2)$, $i=1,2$, to simpler operators $v_i+T_i^0(v_1,v_2)$. Recall the definition of $\mathfrak{B}(\phi,\phi_1)$ in Proposition \ref{le5.2}. We define the operators $I+T_i^\mathfrak{t},\, \mathfrak{t}\in[0,1],~i=1,2$, through the following relations:
\begin{align}
\label{6.1}
\l\nabla(v_1+T_1^\mathfrak{t}(v_1,v_2)),\nabla\phi_1\r=\mathfrak{t}\l\nabla(v_1+T_1(v_1,v_2)),\nabla\phi_1\r+(1-\mathfrak{t})\mathfrak{B}(\phi,\phi_1) \qquad \mathrm{for}~\phi_1\in O_{q,\lambda}^{(1)},
\end{align}
\begin{align}
\label{6.2}
\l\nabla(v_1+T_1^\mathfrak{t}(v_1,v_2)),\nabla\partial_{q}v_{q}\r=\mathfrak{t}\l\nabla(v_1+T_1(v_1,v_2)),\nabla\partial_{q}v_{q}\r+(1-\mathfrak{t})\big(-8\pi\nabla H(q)+8\pi\nabla\psi(q)\big),
\end{align}
\begin{align}
\label{6.3}
\l\nabla(v_1+T_1^\mathfrak{t}(v_1,v_2)),\nabla\partial_{\lambda}v_{q}\r=
\mathfrak{t}\l\nabla(v_1+T_1(v_1,v_2)),\nabla\partial_{\lambda}v_{q}\r-8\pi(1-\mathfrak{t})\Big[2(a-1)\lambda+(\theta-\psi(q))\Big],
\end{align}
\begin{align}
\label{6.4}
\l\nabla(v_1+T_1^\mathfrak{t}(v_1,v_2)),\nabla v_{q}\r
=~&\mathfrak{t}\Big{[}\big(2\lambda+O(1)\big)\l\nabla(v_1+T_1^\mathfrak{t}(v_1,v_2)),\nabla\partial_{\lambda}v_{q}\r\nonumber\\
&+O(1)\|\phi\|_{H^1}+O(\lambda e^{-\lambda})\Big]+16\pi(a-1)\lambda,
\end{align}
while for the second component
\begin{align}
\label{6.5}
v_2+T_2^\mathfrak{t}(v_1,v_2)=\mathfrak{t}(v_2+T_2(v_1,v_2))+(1-\mathfrak{t})\left(w+\psi-\rho_2(-\Delta)^{-1}\left(\frac{h_2e^{w+\psi-8\pi G(x,q)}}
      {\int_Mh_2e^{w+\psi-8\pi G(x,q)}}-1\right)\right),
\end{align}
where the coefficients $O(1)$ in \eqref{6.4} are those terms in (\ref{5.19}) so that $T_1^1(v_1,v_2)=T_1(v_1,v_2)$. We clearly have
$$v_i+T_i(v_1,v_2)=v_i+T_i^1(v_1,v_2), \qquad i=1,2.$$
During the deformation from $T_i^1$ to $T_i^0,~i=1,2$, we have the following result, which will be then used in the analysis of the associated degree.

\begin{lemma}
\label{le6.1}
Assume $\rho_1-8\pi\neq0$, $\rho_2\notin 8\pi\mathbb{N}$ and let $(p,w)$ be a non-degenerate solution of (\ref{sy}). Then, there exists $\varepsilon_1>0$
such that $\bigr(v_1+T_1^\mathfrak{t}(v_1,v_2),v_2+T_2^\mathfrak{t}(v_1,v_2)\bigr)\neq0$ for $(v_1,v_2)\in \partial\big(S_{\rho_1}(p,w)\times S_{\rho_2}(p,w)\big)$ and $\mathfrak{t}\in[0,1]$ if $|\rho_1-8\pi|<\varepsilon_1$ and $\rho_2$ is fixed.
\end{lemma}
\begin{proof}
One should take $(v_1,v_2)\in \bar S_{\rho_1}(p,w)\times \bar S_{\rho_2}(p,w)$, where $\bar S_{\rho_i}(p,w)$ stands for the closure of $S_{\rho_i}(p,w)$, $i=1,2$, such that
$$
\bigr(v_1+T_1^\mathfrak{t}(v_1,v_2),v_2+T_2^\mathfrak{t}(v_1,v_2)\bigr)=0, \qquad \mbox{for some } \mathfrak{t}\in[0,1].
$$
The goal is then to prove that $(v_1,v_2)\notin \partial\big(S_{\rho_1}(p,w)\times S_{\rho_2}(p,w)\big)$. The strategy is to use the estimates in the decomposition (\ref{5.16})-(\ref{5.19}), while for what concerns the estimates for the second component $v_2$ one should exploit also the non-degeneracy of $(p,w)$ to (\ref{sy}).

\medskip

Anyway, the proof of Lemma \ref{le6.1} is quite standard now and we will skip the details: similar arguments can be found in \cite[Lemma~4.3]{cl1}, \cite[Lemma 4.1]{cl3} and \cite[Lemma 5.1]{lwy}.
\end{proof}

\medskip

\noindent The goal is to compute the following degree:
\begin{align} \label{grado}
d_T(p,w)=\mathrm{deg}\Big(\big(v_1+T_1(v_1,v_2),v_2+T_2(v_1,v_2)\big);S_{\rho_1}(p,w)\times S_{\rho_2}(p,w),0\Big).
\end{align}
As we have pointed out, we want to reduce this computation to a finite-dimensional problem (at least for $v_1+T_1(v_1,v_2)$). In order to do this we set
$$S_1^*(p,w)=\Big\{(q,\lambda,a):v_{q,\lambda,a}+\phi\in S_{\rho_1}(p,w), \phi\in O_{q,\lambda}^{(1)}\Big\}$$
and define the map $\Phi_{p}=(\Phi_{p,1},\Phi_{p,2},\Phi_{p,3},\Phi_{p,4})$ by
\begin{align*}
&\Phi_{p,1}=\frac{1}{8\pi}\l\nabla(v_1+T^0_1(v_1,v_2)),\nabla\partial_{q}v_{q}\r,\\
&\Phi_{p,2}=\l\nabla(v_1+T^0_1(v_1,v_2)),\nabla\partial_{\lambda}v_{q}\r,\\
&\Phi_{p,3}=\l\nabla(v_1+T^0_1(v_1,v_2)),\nabla v_{q}\r,\\
&\Phi_{p,4}=v_2+T^0_2(v_1,v_2).
\end{align*}
We notice that due to the decomposition \eqref{dec} and by the fact that $\mathfrak{B}(\phi,\phi_1)$ is positive definite, the projection in the $\phi$ direction does not change the sign of the total degree. Moreover, by Lemma \ref{le6.1} and the invariance of the degree we have
\begin{align}
\label{6.6}
\mathrm{deg}\Big(\big(v_1+&T_1(v_1,v_2),v_2+T_2(v_1,v_2)\big);S_{\rho_1}(p,w)\times S_{\rho_2}(p,w),0\Big)=\mathrm{deg}\Big(\Phi_{p};S_1^*(p,w)\times S_{\rho_2}(p,w),0\Big).
\end{align}
We are now able to compute the right-hand side of (\ref{6.6}) and prove Theorem \ref{th1.4}.\\

\noindent {\em Proof of Theorem \ref{th1.4}.}
We have here to compute the degree $d_T(p,w)$ in \eqref{grado} and prove that is relates to the degree $d_S(p,w)$ of the shadow system \eqref{sy} contributed by the Morse index of $(p,w)$. To do this we compute the right-hand side of (\ref{6.6}).

We start by noting that
\begin{align}
\label{6.7}
\Phi_{p,2}=-\left(\rho_1-8\pi-16\pi\frac{\Delta H(q)}{\rho_1h(q)}\lambda \,e^{-\lambda}\right)
\end{align}
and
\begin{align}
\label{6.8}
\frac{\partial\Phi_{p,1}}{\partial\lambda}=\frac{\partial\Phi_{p,1}}{\partial a}
=\frac{\partial\Phi_{p,2}}{\partial a}=\frac{\partial\Phi_{p,3}}{\partial \psi}
=\frac{\partial\Phi_{p,3}}{\partial q}=\frac{\partial\Phi_{p,4}}{\partial a}
=\frac{\partial\Phi_{p,4}}{\partial\lambda}=0,
\end{align}
It is easy to see that $\Phi_{p,1}=0$, $\Phi_{p,3}=0$ and $\Phi_{p,4}=0$ if and only if
\begin{align}
\label{6.9}
q=p, \quad a=1, \quad \psi=0,
\end{align}
and ${\Phi}_{p,2}=0$ if and only if
\begin{align}
\label{6.10}
\rho_1-8\pi=\frac{16\pi}{\rho_1h(q)}\Delta H(q)\lambda \,e^{-\lambda}.
\end{align}
One can show that if $|\rho_1-8\pi|$ is taken sufficiently small, equation (\ref{6.10}) possesses a unique solution $\lambda=\lambda_1(\rho_1)$. Hence, $(p,\lambda_1(\rho_1),a,0)$ is the only solution of $\Phi_{p}=0,$ where $ a=1.$ To obtain the degree of $\Phi_{p}$
at $(p,\lambda_1(\rho_1),a,0)$ we have to get the number of negative eigenvalues of the following matrix:
\begin{align*}
\mathcal{M}=\left[\begin{array}{llll}
\frac{\partial\Phi_{p,1}}{\partial q}&\frac{\partial\Phi_{p,1}}{\partial \lambda}
&\frac{\partial\Phi_{p,1}}{\partial a}
&\frac{\partial\Phi_{p,1}}{\partial \psi}\vspace{0.2cm}\\
\frac{\partial\Phi_{p,2}}{\partial q}&\frac{\partial\Phi_{p,2}}{\partial \lambda}
&\frac{\partial\Phi_{p,2}}{\partial a}
&\frac{\partial\Phi_{p,2}}{\partial \psi}\vspace{0.2cm}\\
\frac{\partial\Phi_{p,3}}{\partial q}&\frac{\partial\Phi_{p,3}}{\partial \lambda}
&\frac{\partial\Phi_{p,3}}{\partial a}
&\frac{\partial\Phi_{p,3}}{\partial \psi}\vspace{0.2cm}\\
\frac{\partial\Phi_{p,4}}{\partial q}&\frac{\partial\Phi_{p,4}}{\partial \lambda}
&\frac{\partial\Phi_{p,4}}{\partial a}
&\frac{\partial\Phi_{p,4}}{\partial \psi}\vspace{0.2cm}\\
\end{array}\right].
\end{align*}
We point out that $\mu_{\mathcal{M}}$ is an eigenvalue of $\mathcal{M}$ if there exist $\nu\in\mathbb{R}^{2},$
$\lambda,a\in\mathbb{R}$ and $\Psi$ such that
\begin{align*}
\mathcal{M}\left[\begin{array}{l}\nu\\a\\ \lambda\\ \Psi\end{array}\right]=
\mu_{\mathcal{M}}\left[\begin{array}{l}\quad\quad \nu\\ \quad\quad a\\ \quad\quad\lambda\\ (-\Delta)^{-1}\Psi\end{array}\right].
\end{align*}
We set $N(T)$ as the number of the negative eigenvalues (with multiplicity) of the matrix $T$. Let
\begin{align*}
\mathcal{M}_1=\left[\begin{array}{ll}\frac{\partial\Phi_{p,1}}{\partial p}&\frac{\partial\Phi_{p,1}}{\partial \psi}\vspace{0.2cm}\\
\frac{\partial\Phi_{p,4}}{\partial p}&\frac{\partial\Phi_{p,4}}{\partial \psi}\\
\end{array}\right] \quad
\mathrm{and} \quad
\mathcal{M}_2=\left[\begin{array}{ll}\frac{\partial\Phi_{p,2}}{\partial \lambda}&\frac{\partial\Phi_{p,2}}{\partial a}\vspace{0.2cm}\\
\frac{\partial\Phi_{p,3}}{\partial \lambda}&\frac{\partial\Phi_{p,3}}{\partial a}\\
\end{array}\right].
\end{align*}
By using (\ref{6.8}) we conclude that
\begin{align*}
N(\mathcal{M})=N(\mathcal{M}_1)+N(\mathcal{M}_2),
\end{align*}
or equivalently
\begin{align*}
\mathrm{sgn}\bigr(\det( \mathcal{M})\bigr)&=\mathrm{sgn}\bigr(\det( \mathcal{M}_1)\bigr)\;\;\mathrm{sgn}\bigr(\det( \mathcal{M}_2)\bigr) \\
&= \mathrm{sgn}\bigr(\det( \mathcal{M}_1)\bigr)\;\;\mathrm{sgn}\Big(\frac{\partial \Phi_{p,2}}{\partial\lambda}\Big)\;\;\mathrm{sgn}
\Big(\frac{\partial \Phi_{p,3}}{\partial a}\Big).
\end{align*}
Therefore,
\begin{align}\label{gr}
\deg\Big(\Phi_p;S_1^*(p,w)\times &S_{\rho_2}(p,w),0\Big)=(-1)^{N(\mathcal{M}_1)}\;\;\mathrm{sgn}\Big(\frac{\partial \Phi_{p,2}}{\partial\lambda}\Big)\;\;\mathrm{sgn}
\Big(\frac{\partial \Phi_{p,3}}{\partial a}\Big).
\end{align}

\medskip

\noindent First, by its definition it is easy to see that
$$
\mathrm{sgn}
\Big(\frac{\partial \Phi_{p,3}}{\partial a}\Big)=1.
$$
To compute $\frac{\partial\Phi_{p,2}}{\partial\lambda}$, recall that we are considering $q=p$. We have
$$\frac{\partial\Phi_{p,2}}{\partial\lambda}=-\frac{16\pi}{\rho_1h(p)}\Delta H(p)\lambda \,e^{-\lambda}+O(e^{-\lambda}).$$
Thus, by \eqref{6.10} we deduce
\begin{align*}
\frac{\partial\Phi_{p,2}}{\partial\lambda}=-(\rho_1-8\pi)+O(e^{-\lambda}).
\end{align*}

\medskip

\noindent Up to now we got from \eqref{gr}
\begin{align*}
\deg\Big(\Phi_p;S_1^*(p,w)\times &S_{\rho_2}(p,w),0\Big)=-\mathrm{sgn}(\rho_1-8\pi)\,(-1)^{N(\mathcal{M}_1)}.
\end{align*}It remains to compute $N(\mathcal{M}_1).$ One has $\frac{\partial\Phi_{p,1}}{\partial \psi}[\Psi]=\nabla\Psi(p)$ and
\begin{align*}
\frac{\partial\Phi_{p,4}}{\partial \psi}[\Psi]=\Psi-(-\Delta)^{-1}\left(\left(\rho_2\frac{h_2e^{w-8\pi G(x,p)}}
{\int_Mh_2e^{w-8\pi G(x,p)}}\right)\Psi-\rho_2\frac{h_2e^{w-8\pi G(x,p)}}
{\big(\int_Mh_2e^{w-8\pi G(x,p)}\big)^2}
\int_M\big(h_2e^{w-8\pi G(x,p)}\Psi\big)\right).
\end{align*}
Therefore, we deduce
\begin{align}
\label{6.11}
\mathcal{M}_1 \left(\begin{array}{l}
\nu\\
\Psi
\end{array}
\right) =
\Big[\frac{\partial(\Phi_{p,1},\Phi_{p,4})}{\partial(p,\psi)}\Big]
\left(\begin{array}{l}
\nu\\
\Psi
\end{array}
\right)
=\left(\begin{array}{l}
-\nabla^2H(p)\cdot\nu+\nabla\Psi(p)\\
\quad\quad\quad-\mathcal{I}_0
\end{array}\right),
\end{align}
where
\begin{align*}
\mathcal{I}_0=&-\Psi+(-\Delta)^{-1}\left(\rho_2\frac{h_2e^{w-8\pi G(x,p)}}
{\int_Mh_2e^{w-8\pi G(x,p)}}\Psi-\rho_2\frac{h_2e^{w-8\pi G(x,p)}}
{\big(\int_Mh_2e^{w-8\pi G(x,p)}\big)^2}\int_M\big(h_2e^{w-8\pi G(x,p)}\Psi\big) \right.\\
&-8\pi\rho_2\frac{h_2e^{w-8\pi G(x,p)}}
{\int_Mh_2e^{w-8\pi G(x,p)}}\big(\nabla G(x,p)\cdot \nu\big)\\
&+8\pi\rho_2\frac{h_2e^{w-8\pi G(x,p)}}
{\big(\int_Mh_2e^{w-8\pi G(x,p)}\big)^2}
\int_M\Big[h_2e^{w-8\pi G(x,p)}\big(\nabla G(x,p)\cdot \nu\big)\Big]\Biggr).
\end{align*}
We observe that \eqref{6.11} coincides with the eigenvalue problem of the linearized equation of (\ref{sy}) around the solution $(p,w)$. Thus, we get that ${N(\mathcal{M}_1)}$ is exactly the number of the negative eigenvalues of the linearized equation of (\ref{sy}), namely $(-1)^{N(\mathcal{M}_1)}=d_S(p,w)$, the degree of the shadow system \eqref{sy} contributed by the solution $(p,w)$. Therefore, we conclude that
$$
d_T(p,w)=-\mathrm{sgn}(\rho_1-8\pi)\,d_S(p,w).
$$
This concludes the proof of the Theorem~\ref{th1.4}.
\begin{flushright}
$\square$
\end{flushright}

\medskip

\noindent As a consequence, we can state the following:

\medskip

\noindent {\em Proof of Theorem \ref{bubb}.} Theorem \ref{bubb} follows from the Theorem \ref{th1.4}, see also the discussion in the Introduction and at the beginning of the Section \ref{sec:operator}.
\begin{flushright}
$\square$
\end{flushright}

\medskip

\noindent The next step is to compute the total degree of the shadow system (\ref{sy}). The strategy will be to decouple the system \eqref{sy} and then to use Theorem A to get the degree of the first equation in \eqref{sy}. In order to decouple the system and to simplify the problem we introduce the following deformation:
\begin{equation}
\label{6.12}
(S_\mathfrak{t})
\left\{\begin{array}{l}
\Delta w+\rho_2\left(\frac{h_2e^{w-8\pi G(x,p)}}
{\int_Mh_2e^{w-8\pi G(x,p)}}-1\right)=0, \vspace{0.2cm}\\
\nabla\Big(\log(h_1e^{-w\cdot(1-\mathfrak{t})})+4\pi R(x,x)\Big)_{|{x=p}}=0,
\end{array}\right. \qquad \mathfrak{t}\in[0,1].
\end{equation}
Clearly, we are starting from the system defined in \eqref{s} and we end up with a decoupled system. During the deformation
from $(S_1)$ to $(S_0)$ we have the following result, which will be then used in the degree analysis.
\begin{lemma}
\label{le6.2}
Let $\rho_2\notin8\pi\mathbb{N}$. Then there exists a uniform constant $C_{\rho_2}$ such that for all solutions to (\ref{6.12})
we have $|w|_{L^{\infty}(M)}<C_{\rho_2}.$
\end{lemma}

\begin{proof}
Since $\rho_2\notin8\pi\mathbb{N}$, by classic results concerning the blow up analysis of equation \eqref{eq2}, see \cite{bt}, any solution of
\begin{equation}
\label{6.13}
\Delta w+\rho_2\left(\frac{h_2e^{w-8\pi G(x,p)}}{\int_Mh_2e^{w-8\pi G(x,p)}}-1\right)=0
\end{equation}
is uniformly bounded above. The proof of the lemma follows then by using classical elliptic estimates.
\end{proof}

\medskip

\noindent{\em Proof of Theorem \ref{th1.5}.} Since the topological degree is independent of $h_1$ and $h_2$, by the Theorem \ref{th2.1.2} we can always choose $h_1$ and $h_2$ such that the solutions to the shadow system (\ref{sy}) are non-degenerate.

Let $d_{sy}$ denote the Leray-Schauder degree for (\ref{sy}). By Lemma \ref{le6.2} and the invariance of the degree, we have just to compute the topological degree of (\ref{6.12}) when $\mathfrak{t}=1,$ namely
\begin{align}
\label{6.14}
\left\{\begin{array}{l}
\Delta w+\rho_2\left(\frac{h_2e^{w-8\pi G(x,p)}}{\int_Mh_2e^{w-8\pi G(x,p)}}-1\right)=0,\vspace{0.2cm}\\
\nabla\bigr(\log h_1+4\pi R(x,x)\bigr)_{|x=p}=0.
\end{array}\right.
\end{align}
Since this is a decoupled system, the topological degree is given by the product of the degree of first equation and the degree contributed by the second equation. By the Poincare-Hopf Theorem, the degree of the second equation is simply $\chi(M)$, i.e. the Euler characteristic of $M$. On the other hand, using Theorem A with $|S|=1$ and $\alpha_q=2$ (see also Remark \ref{remark}), the topological degree for the first equation is $b_k+b_{k-1}+b_{k-2},$ where $b_{k}$ is given (\ref{b_k}). Therefore,
\begin{align}
\label{6.15}
d_{sy}=\chi(M)\bigr(b_k+b_{k-1}+b_{k-2}\bigr).
\end{align}
This concludes the proof of Theorem \ref{th1.5}.
\begin{flushright}
$\square$
\end{flushright}

\

\noindent Finally, we are now in position to prove the main Theorem \ref{th1.6}.

\

\noindent {\em Proof of Theorem \ref{th1.6}.} Theorem \ref{th1.6} is a consequence of Theorems \ref{th1.4}, \ref{th1.5}, \ref{th4.1} and Theorem A.

Firs of all, in order to apply these results, since the topological degree is independent of $h_1$ and $h_2$, by Remark \ref{rem:h} we can always choose $h_1$ and $h_2$ such that both the solutions to the shadow system (\ref{sy}) are non-degenerate and $l(p)\neq 0$, where $l(p)$ is given in \eqref{l(p)}.

\medskip

Using the notation introduced in \eqref{not} we have to prove that
$$
d_{SG}(2) = b_k-\chi(M)\bigr(b_k+b_{k-1}+b_{k-2}\bigr).
$$
As discussed in the Introduction, we know that
$$
	d_{SG}(2) = d_{SG}(1) ~+ ~\Bigr\{\mbox{degree of the blow up solutions for }\rho_1~\mathrm{crosses}~8\pi\Bigr\}.
$$
Since the degree of the bounded solutions stays constant when $\rho_1$ crosses $8\pi$, the degree jump is due to the blow up solutions for $\rho_1=8\pi$ in the following way:
\begin{align}\label{jump}
\begin{split}
	d_{SG}(1) -d_- =d_{SG}(2) - d_+,
\end{split}
\end{align}
where $d_-,d_+$ stands for the degree contributed by the bubbling solutions when $\rho_1\rightarrow8\pi^{-}$ and $\rho_1\rightarrow8\pi^+$ respectively. By Theorem \ref{th4.1} we know that all the blow up solutions are contained in the set $S_{\rho_1}(p,w)\times S_{\rho_2}(p,w)$ for some solution $(p,w)$ of (\ref{sy}) such that $l(p)\neq 0$. Moreover, the degree of each of these blow up solutions is given by Theorem \ref{th1.4}. Furthermore, by Remark \ref{rem:segno} we know that
$$
	\mbox{sgn}(\rho_{1}-8\pi) = \mbox{sgn}(l(p)).
$$
Therefore, from \eqref{jump} we deduce
$$
d_{SG}(1) - \sum_{l(p)<0} d_S(p,w) = d_{SG}(2) + \sum_{l(p)>0} d_S(p,w),
$$
hence the jump is given by
$$
d_{SG}(2)-d_{SG}(1)= -\sum_{l(p)\neq 0} d_S(p,w) = -d_S,
$$
where $d_S$ is the total degree of the shadow system \eqref{sy}. By Theorem \ref{th1.5} we get
$$
d_{SG}(2)-d_{SG}(1)=\chi(M) \bigr(b_k+b_{k-1}+b_{k-2}\bigr),
$$
where $b_k$ is defined in \eqref{b_k}. Since by Theorem A $d_{SG}(1)=b_k$, the proof of Theorem \ref{th1.6} is concluded.
\begin{flushright}
$\square$
\end{flushright}

\vspace{1cm}
\section{Appendix: proof of Proposition \ref{le5.2}}
In this section we will give the proof of Proposition \ref{le5.2} which is based on the decomposition of $\Delta\big(v_1+T_1(v_1,v_2)\big)$ in (\ref{5.13})-(\ref{5.15}). We follow here \cite{cl1, lwy}. Let
\begin{align*}
\overline{v}:=&\frac{\int_{B_{r_0}(0)}\frac{e^{\lambda}}{(1+e^{\lambda}|y|^{2})^2}v(y)\,\mathrm{d}y}
{\int_{\mathbb{R}^2}\frac{e^{\lambda}}{(1+e^{\lambda}|y|^{2})^2}\,\mathrm{d}y}
=\frac{1}{\pi}\int_{B_{r_0}(0)}\frac{e^{\lambda}}{(1+e^{\lambda}|y|^{2})^2}v(y)\,\mathrm{d}y.
\end{align*}
We start by pointing out the following Poincare-type inequality:
\begin{align}
\label{7.1}
\int_{B_{r_0}(0)}\frac{e^{\lambda}}{(1+e^{\lambda}|y|^{2})^2}\phi^2(y)\,\mathrm{d}y
\leq c\bigr(\|\phi\|_{H^1(B_{r_0}(0))}^2+\overline{\phi}^2\bigr),
\end{align}
for some constant $c=c(r_0)$ independent of $\lambda$.

\bigskip

Concerning the part which contains $E$, see \eqref{5.11} we let $\varepsilon_2>0$ be small, which will be chosen later. Write
$$E=E^++E^-,$$
with
\begin{align*}
E^+=\left\{\begin{array}{ll}
E~&\mathrm{if}~|\varphi|\geq\varepsilon_2,\\
0~&\mathrm{if}~|\varphi|<\varepsilon_2,
\end{array}\right.
\quad\mathrm{and}\quad
E^-=\left\{\begin{array}{ll}
0~&\mathrm{if}~|\varphi|\geq\varepsilon_2,\\
E~&\mathrm{if}~|\varphi|<\varepsilon_2.
\end{array}\right.
\end{align*}
As $\lambda\rightarrow\infty,$ we have
$$E^+=O(e^{|\varphi|+\lambda})$$
and
$$E^-=\rho_1h(q)\,e^{U}(O(\varphi^2)+O(\beta^2)).$$

\bigskip

Recall now $v_1 \in S_{\rho_1}(p,w)$ (see \eqref{4.14}) is in the form $v_1=v_{q,\lambda,a}+\phi,$ $\phi\in O_{q,\lambda}^{(1)}$ (see \eqref{4.9}). We have the following result.

\begin{lemma} (\cite{cl1})
\label{le7.1}
Let $U(x)$ and $\sigma$ be defined as in (\ref{4.3}) and \eqref{sigma}, respectively. Assume $\phi\in O_{q,\lambda}^{(1)}$. Then there is a constant $c$ and $\epsilon>0$ such that for large $\lambda$ it holds
\begin{equation}
\label{7.2}
\int_{B_{r_0}(q)}e^{U}\phi\,\mathrm{d}y=O(\lambda^2e^{-\lambda}\|\phi\|_{H^1}),
\end{equation}
and
\begin{equation}
\label{7.3}
\int_M\Big(|\nabla\phi|^2-\rho_1h(q)\,e^U\sigma(x)\phi^2\Big)\geq c\int_M|\nabla\phi|^2.
\end{equation}
\end{lemma}

\vspace{0.5cm}

\noindent{\em Proof of Proposition \ref{le5.2}}. The proof is based on the proof of \cite[Lemma 4.2]{cl1} and \cite[Lemma 4.4]{lwy} and we shall sketch the process here. We start with part (1). Let $\phi\in O_{q,\lambda}^{(1)}$ and $\psi\in O_{q,\lambda}^{(2)}.$ Recall $v_1=v_{q,\lambda,a}+\phi,$ $\phi\in O_{q,\lambda}^{(1)}$ and $v_2=w+\psi,$ $\psi\in O_{q,\lambda}^{(2)}.$ In order to get the estimates we consider
$$\l\nabla\big(v_1+T_1(v_1,v_2)\big),\nabla \phi_1\rangle=-\l\Delta\big(v_1+T_1(v_1,v_2)\big),\phi_1\rangle.$$
Recall the decomposition of $\Delta\big(v_1+T_1(v_1,v_2)\big)$ in (\ref{5.13})-(\ref{5.15}). We write
\begin{align}
\label{7.4}
\l\nabla\big(v_1+T_1(v_1,v_2)\big),\nabla\phi_1\r=&\int\nabla\phi\cdot\nabla\phi_1
-\int_{B_{r_0}(q)}\rho_1h(p)\,e^{U}\phi\,\phi_1+\mathrm{remainder~terms}
\nonumber\\
:=&\mathfrak{B}(\phi,\phi_1)+\mathrm{remainder~terms}.
\end{align}
Clearly, $\mathfrak{B}$ is a symmetric bilinear form in $O_{q,\lambda}^{(1)}$ and by the second part of the Lemma \ref{le7.1}, $\mathfrak{B}(\phi,\phi)\geq c_0\|\phi\|_{H^1(M)}^2$ for some $c_0>0.$
For the remainder terms, recalling that $\int_M \phi = 0$ and by the first part of the Lemma \ref{le7.1} we deduce
\begin{align*}
\int_M\big(8\pi(a-1)+(8\pi-\rho_1)\big)\phi_1=0,
\end{align*}
\begin{align}
\label{7.5}
\Big|\int_{B_{r_0}(q)}\rho_1h(p)e^{U}\phi_1\Big|=O(\lambda^2e^{-\lambda})\|\phi_1\|_{H^1(M)}.
\end{align}
Still by using Lemma \ref{le7.1} and by \eqref{5.8} we conclude that the term concerned with $\frac{e^t}{\int_M h_1 e^{v_1-v_2}}-1$ is small. Using $|\nabla H(q)|\leq C\lambda \,e^{-\lambda}$ for $v_1\in S_{\rho_1}(p,w)$ we obtain
\begin{align}
\label{7.6}
\int_{B_{r_0}(q)}\nabla H\cdot (x-q)\rho_1h(q)\,e^{U}\phi_1=O(\lambda \,e^{-\lambda})\|\phi_1\|_{H^1(M)}.
\end{align}
By Lemma \ref{le7.1}, we have
\begin{align}
\label{7.7}
\int_{B_{r_0}(q)}\rho_1h(q)\,e^{U}(a-1)(U+s-1)\phi_1
=O(|a-1|)\|\phi_1\|_{H^1(M)}=O(\lambda \,e^{-\lambda})\|\phi_1\|_{H^1(M)}.
\end{align}

For $E^+$ and $E^-$ we obtain
\begin{align}
\label{7.8}
\int_{B_{r_0}(q)}|E^+\phi_1|\leq~\left(\int_{B_{r_0}(p)}|E^+|^2\right)^{\frac12}
\left(\int_{B_{r_0}(q)}\phi_1^2\right)^{\frac{1}{2}}=~O(\lambda \,e^{-\lambda})\|\phi_1\|_{H^1(M)},
\end{align}
and
\begin{align}
\label{7.9}
\int_{B_{r_0}(q)}|E^-\phi_1|\leq&\int_{B_{r_0}(q)}\rho_1h(q)\,e^{U}(O(\varphi^2)+O(\beta^2))\phi_1\nonumber\\
=&~O(\varepsilon_2)(\|\phi\|_{H^1(M)}+\lambda \,e^{-\lambda})\|\phi_1\|_{H^1(M)}+O\left(\frac{\lambda^3}{e^{2\lambda}}\right)
\|\phi_1\|_{H^1(M)},
\end{align}
provided $\varepsilon_2$ is small.

For the term which involves $\psi$, we have
\begin{align}
\label{7.10}
\Big|\int_{B_{r_0}(q)}\rho_1h(q)\,e^{U}\phi_1\psi\Big|=&~\Big|\int_{B_{r_0}(q)}\rho_1h(q)\,e^{U}\phi_1(\psi-\psi(q))\Big|+
\Big|\int_{B_{r_0}(q)}\rho_1h(q)\,e^{U}\phi_1\psi(q)\Big|\nonumber\\
=&~O(\lambda \,e^{-\lambda})\|\phi_1\|_{H^1(M)}.
\end{align}

\medskip

\noindent We consider now the terms in $M\setminus B_{r_0}(q)$. By \cite[Lemma 2.2]{cl1} we get
\begin{align}
\label{7.11}
\int_{B_{2r_0}(q)\setminus B_{r_0}(q)}\Delta(v_q-8\pi G(x,q))\phi_1=O\left(\frac{\lambda}{e^{\lambda}}\right)\|\phi_1\|_{H^1(M)}.
\end{align}
For the nonlinear term in $\Delta T_1(v_1,v_2)$ on $M\setminus B_{r_0}(q)$ we have
$$\int_{M\setminus B_{r_0}(q)}|e^{\varphi}\phi_1|=O\left(\int_{|\varphi|\geq\varepsilon_2}|e^{\varphi}\phi_1|
+\int_{|\varphi|\leq\varepsilon_2}|e^{\varepsilon_2}\phi_1|\right)=O(1)\|\phi_1\|_{H^1(M)}.$$
Because $\int_{M}h_1e^{v_1-v_2}\sim e^{\lambda},$ we deduce
\begin{align}
\label{7.12}
\int_{M\setminus B_{r_0}(q)}\frac{\rho_1h_1e^{v_1-v_2}}{\int_{M}h_1e^{v_1-v_2}}|\phi_1|=O(e^{-\lambda})
\int_{M\setminus B_{r_0}(q)}e^{\varphi}|\phi_1|=O(e^{-\lambda})\|\phi_1\|_{H^1(M)}.
\end{align}
It is easy to see that the remaining parts in $M\setminus B_{r_0}(q)$ are small.

Combining (\ref{7.4})-(\ref{7.12}), we get
\begin{align*}
\l\nabla(v_1+T_1(v_1,v_2)),\nabla\phi_1\r=\l\nabla\phi,\nabla\phi_1\r-\int_{B_{r_0}(q)}\rho_1h(q)e^{U}\phi\phi_1+O(\lambda \,e^{-\lambda})
\|\phi_1\|_{H^1(M)}.
\end{align*}

\bigskip

We prove now part (3). First, we note that by the definition of $v_q$, see \eqref{4.8}, and by Lemma \ref{le4.1}:
\begin{align}
\label{7.13}
\partial_{\lambda}v_q=\left(2-\frac{\frac{\rho_1h(q)}{4}e^{\lambda}|x-q|^2}{1+\frac{\rho_1h(q)}{8}e^{\lambda}|x-q|^2}+
\partial_{\lambda}\left[\eta+\frac{2\Delta H(q)}{\rho_1h(q)}\lambda^2e^{-\lambda}\right]\right)\sigma=(1+\partial_{\lambda}U)\sigma+O(\lambda^2e^{-\lambda}).
\end{align}
Since $\phi\in O_{q,\lambda}^{(1)},$ we have $\int_{M}\nabla\phi\cdot\nabla\partial_{\lambda}v_q=0$. It is not difficult to get
\begin{align}
\label{7.14}
\int_{B_{r_0}(q)}\partial_{\lambda}v_q=\int_{B_{r_0}(q)}(1+\partial_{\lambda}U)+O(\lambda^2e^{-\lambda})
=O(\lambda^2e^{-\lambda}).
\end{align}
Then
\begin{align}
\label{7.15}
\big(8\pi(a-1)+8\pi-\rho_1\big)\int_{B_{r_0}(q)}\partial_{\lambda}v_q=O(\lambda^3e^{-2\lambda}).
\end{align}
Again by (\ref{7.13}), we have
\begin{align}
\label{7.16}
\int_{B_{r_0}(q)}\rho_1h(q)\,e^{U}\partial_{\lambda}v_q=8\pi+O(\lambda^2e^{-\lambda})
\end{align}
and
\begin{align}
\label{7.17}
\int_{B_{r_0}(q)}&\rho_1h(q)\,e^{U}\left[-2\log\left(1+\frac{\rho_1h(q)}{8}e^{\lambda}|x-q|^2\right)\right]
\partial_{\lambda}v_q=-8\pi+O\left(\frac{\lambda^2}{e^{\lambda}}\right).
\end{align}
Combining (\ref{7.16}) and (\ref{7.17}) we deduce
\begin{align}
\label{7.18}
\int_{B_{r_0}(q)}&\rho_1h(q)\,e^{U}(U+s-1)\partial_{\lambda}v_q
=16\pi\lambda-16\pi+16\pi\log\frac{\rho_1h(q)}{8}+64\pi^2R(q,q)+O(\lambda \,e^{-\lambda}).
\end{align}
Using a scaling argument it is possible to show that
\begin{align}
\label{7.19}
\int_{B_{r_0}(q)}|a-1|\rho_1h(q)\,e^{U}O(|x-q|)\partial_{\lambda}v_q=O(e^{-\frac12\lambda})|a-1|=O(\lambda \,e^{-\frac32\lambda}).
\end{align}
and
\begin{align}
\label{7.20}
\int_{B_{r_0}(q)}&\rho_1h(q)\,e^{U}\nabla H(q)\cdot(x-q)\partial_{\lambda}v_q
=O\left(\frac{\lambda^2}{e^{\lambda}}\right)\int_{B_{r_0}(q)}e^{U}|x-q|=O(\lambda^2e^{-\frac32\lambda}),
\end{align}
where we have used that $\nabla H(q)\cdot(x-q)$ is an odd function.

Next, we estimate the term $\phi\partial_{\lambda}v_q$ and $\psi\partial_{\lambda}v_q$. By notice that
\begin{align}
\label{7.21}
0=\int_{M}\nabla\phi\cdot\nabla\partial_{\lambda}v_q=&-\int_{M}\phi\Delta(\partial_{\lambda}v_q)
=\rho_1h(q)\int_{B_{r_0}(q)}e^{U}\phi\,\partial_{\lambda}U+O(\lambda \,e^{-\lambda}\|\phi\|_{H^1(M)}).
\end{align}
Hence, by Lemma \ref{le7.1} and (\ref{7.21}), we have
\begin{align}
\label{7.22}
\int_{B_{r_0}(q)}\rho_1h(q)e^{U}\phi\,\partial_{\lambda}v_q=O(\lambda^2e^{-\lambda})\|\phi\|_{H^1(M)}=O(\lambda^3e^{-2\lambda}),
\end{align}
and
\begin{align}
\label{7.23}
\int_{B_{r_0}(q)}\rho_1h(q)\,e^{U}\psi\partial_{\lambda}v_q=8\pi\psi(q)+O(\lambda e^{-\frac32\lambda}).
\end{align}
By (\ref{7.13}), Lemma \ref{le7.1} and the Moser-Trudinger inequality,
\begin{align}
\label{7.24}
\int_{B_{r_0}(q)}|E^+\partial_{\lambda}v_q|\leq O(e^{-2\lambda})~\mathrm{and}~\int_{B_{r_0}(q)}|E^-\partial_{\lambda}v_q|=O(\lambda^3e^{-2\lambda}).
\end{align}

\medskip

Consider now $M\setminus B_{r_0}(q)$ and observe that by definition $\partial_{\lambda}v_q=0$ on $M\setminus B_{2r_0}(q)$. On the other hand in $B_{r_0}(q)\setminus B_{r_0}(q)$ we have the following estimates:
$$e^{v_1-v_2}=O(e^{\phi}), \quad \frac{\rho_1h_1e^{v_1-v_2}}{\int_{M}h_1e^{v_1-v_2}}=O(e^{-\lambda})\,e^{\phi},
\quad \partial_{\lambda}v_q=O\left(\frac{\lambda^2}{e^{\lambda}}\right).$$
By the Moser-Trudinger inequality,
\begin{align}
\label{7.25}
\int_{M\setminus B_{r_0(q)}}\frac{\rho_1h_1e^{v_1-v_2}}{\int_{M}h_1e^{v_1-v_2}}\partial_{\lambda}v_q
=O(\lambda^3e^{-2\lambda}).
\end{align}
By \cite[Lemma 2.2]{cl1} and $\partial_{\lambda}v_q=O(\lambda^2e^{-\lambda}),$ we have
\begin{align}
\label{7.26}
\int_{B_{2r_0}(q)\setminus B_{r_0}(q)}\Delta(v_q-\overline{v}_q-8\pi G(x,q))\cdot\partial_{\lambda}v_q=O(\lambda^3e^{-2\lambda}).
\end{align}
It is easy to see that the remaining parts in $M\setminus B_{r_0}(q)$ are small.

\medskip

Combining (\ref{7.13}) to (\ref{7.26}), we obtain
\begin{align}
\label{7.27}
\l\nabla(v_1+T_1(v_1,v_2)),\nabla\partial_{\lambda}v_q\r
=&-(a-1)\Big{(}16\pi\lambda-16\pi+16\pi\log\frac{\rho_1h(q)}{4}+64\pi^2R(q,q)\Big{)}\nonumber\\
&-8\pi\Big(\frac{e^{t}}{\int_{M}h_1e^{v_1-v_2}}-1\Big)+8\pi\psi(q)+O(\lambda e^{-\frac32\lambda}).
\end{align}
This proves part (3).

\bigskip

We consider now part (4). We start by observing that
\begin{align*}
\l\nabla(v_1+T_1(v_1,v_2)),\nabla v_q\r=\l\nabla(v_1+T_1(v_1,v_2)),\nabla(v_q-\overline{v}_q)\r=-\l\Delta(v_1+T_1(v_1,v_2)),(v_q-\overline{v}_q)\r.
\end{align*}
Note that
\begin{equation*}
\int_M\big[8\pi(a-1)+8\pi-2\rho_1\big](v_q-\overline{v}_q)=0 \quad \mathrm{and} \quad
\int_{M}\nabla\phi\cdot\nabla(v_q-\overline{v}_q)=0.
\end{equation*}
On $B_{r_0}(p)$ we have
\begin{align}
\label{7.28}
v_q-\overline{v}_q=2\lambda-2\log(1+\frac{\rho_1h(q)}{8}\,e^{\lambda}|x-q|^2)+8\pi R(q,q)+O(|y|)+2\log\frac{\rho_1h(q)}{8}+O\left(\frac{\lambda^2}{e^{\lambda}}\right).
\end{align}
We use (\ref{7.28}) to compute $\int_{M}\rho_1h\,e^{U}(U+s-1)(v_q-\overline{v}_q)$. After a scaling argument one can show that
\begin{align}
\label{7.29}
\int_{B_{r_0}(q)}\rho_1h(q)\,e^{U}\log\left(1+\frac{\rho_1h(q)}{8}e^{\lambda}|x-q|^2\right)=
8\pi+O\left(\frac{\lambda}{e^{\lambda}}\right),
\end{align}
\begin{align}
\label{7.30}
\int_{B_{r_0}(q)}\rho_1h(q)e^{U}\Big{[}\log(1+\frac{\rho_1h(q)}{8}e^{\lambda}|x-q|^2)\Big{]}^2
=16\pi+O\left(\frac{\lambda^2}{e^{\lambda}}\right),
\end{align}
and
\begin{align}
\label{7.31}
\int_{B_{r_0}(q)}\rho_1h\,e^{U}\left[\log\left(1+\frac{\rho_1h(q)}{4}e^{\lambda}|x-q|^2\right)\right]O(|x-q|)=O(e^{-\frac12\lambda}).
\end{align}
Therefore, by (\ref{7.29})-(\ref{7.31}),
\begin{align}
\label{7.32}
\int_{B_{r_0}(q)}\rho_1h\,e^{U}(U+s-1)(v_q-\overline{v}_q)=
\Big{[}256\pi^2R(q,q)+64\pi\log\frac{\rho_1h(q)}{8}-16\pi\Big{]}\lambda+32\pi\lambda^2-64\pi\lambda+O(1).
\end{align}
Similarly, we have
\begin{align}
\label{7.33}
\int_{B_{r_0}(q)}\rho_1h(q)\,e^{U}(v_q-\overline{v}_q)=
16\pi\lambda-16\pi+64\pi^2R(q,q)+16\pi\log\frac{\rho_1h(q)}{4}+O(e^{-\frac12\lambda}),
\end{align}
and
\begin{align}
\label{7.34}
\int_{B_{r_0}(q)}\rho_1h(q)\,e^{U}O(|x-q|)(v_q-\overline{v}_q)=O(\lambda e^{-\frac12\lambda}).
\end{align}
We have $\nabla H(q)=O(\lambda \,e^{-\lambda})$. By (\ref{7.28}) and by the fact that $\nabla H(q)\cdot (x-q)$ is an odd function, we deduce
\begin{align}
\label{7.35}
\int_{B_{r_0}(q)}\rho_1h(q)\,e^{U}\nabla H(q)\cdot (x-q)(v_q-\overline{v}_q)=O(\lambda^2e^{-2\lambda}).
\end{align}
By Lemma \ref{le7.1} we get
\begin{align}
\label{7.36}
\int_{B_{r_0}(q)}\rho_1h(q)\,e^{U}\varphi(v_q-\overline{v}_q)
=&\int_{B_{r_0}(q)}\rho_1h(q)\,e^{U}\phi\left[\lambda+s-2\log\left(1+\frac{\rho_1h(q)}{8}e^{\lambda}|x-q|^2\right)+O(|x-q|)\right]
\nonumber\\
&-\int_{B_{r_0}(q)}\rho_1h(q)\,e^{U}\psi\left[\lambda+s-2\log\left(1+\frac{\rho_1h(q)}{8}e^{\lambda}|x-q|^2\right)+O(|x-q|)\right]
\nonumber\\
=~&O(1)\|\phi\|_{H^1(M)}+o(1)\|\psi\|_*-16\pi\lambda\psi(q).
\end{align}
Similarly as in the proof of part (3), we have $\int_{B_{r_0}(q)}E(v_q-\overline{v}_q)=O(\lambda^4e^{-2\lambda}).$
Since $v_q=O(1)$ on $M\setminus B_{r_0}(q),$ by \cite[Lemma 2.2]{cl1},
\begin{equation}
\label{7.37}
\int_{B_{2r_0}(q)\setminus B_{r_0}(q)}\Delta(v_q-8\pi G(x,q))(v_q-\overline{v}_q)=O(\lambda \,e^{-\lambda}).
\end{equation}

\medskip

We focus now on the integral outside $B_{r_0}(q).$ We have $(\int_{M}h_1e^{v_1-v_2})^{-1}=O(e^{-\lambda})$ and
\begin{align}
\label{7.38}
\int_{B_{2r_0}(q)\setminus B_{r_0}(q)}\frac{\rho_1h_1}{\int_{M}h_1\,e^{v_1-v_2}}e^{v_1-v_2}(v_q-\overline{v}_q)=\int_{B_{2r_0}(q)\setminus B_{r_0}(q)}O(e^{-\lambda})\,e^{\phi-\psi}=O(e^{-\lambda}).
\end{align}
Similarly we can prove
\begin{equation}
\label{7.39}
\int_{M\setminus B_{2r_0}(q)}\frac{\rho_1h_1}{\int_{M}h_1e^{v_1-v_2}}e^{v_1-v_2}(v_q-\overline{v}_q)=O(e^{-\lambda}).
\end{equation}
It is easy to see that the remaining parts in $M\setminus B_{r_0}(q)$ are small.

\medskip

By (\ref{7.32})-(\ref{7.39}), we have
\begin{align}
\label{7.40}
&\l\nabla(v_1+T_1(v_1,v_2),\nabla(v_q-\overline{v}_q)\r\nonumber\\
=~&(2\lambda-2+8\pi R(q,q)+2\log\frac{\rho_1h(q)}{8})\l\nabla(v_1+T_1(v_1,v_2)),\nabla\partial_{\lambda}v_q\r\nonumber\\
&+16\pi(a-1)\lambda+O(1)\|\phi\|_{H^1(M)}+o(1)\|\psi\|_*+O(\lambda \,e^{-\lambda}).
\end{align}

\bigskip

We conclude now with the proof of part (2). We observe that
$$\l\nabla(v_1+T_1(v_1,v_2)),\nabla \partial_{q} v_q\r=\l\nabla\partial_{q}(v_1+T_1(v_1,v_2)),\nabla(v_q-\overline{v}_q)\r.$$
Since $\phi\in O_{q,\lambda}^{(1)}$ we get $\l\nabla\phi,\nabla\partial_{q}(v_q-\overline{v}_q)\r=0.$
Moreover, using $\int_M(v_q-\overline{v}_q)=0$ we obtain $$\int_M(8\pi(a-1)+8\pi-2\rho_1)\partial_q(v_q-\overline{v}_q)=0.$$
For $x\in B_{r_0}(q),$ by \cite[Lemma 2.1]{cl1} one has
\begin{align}
\label{7.41}
\partial_{q}v_q=&-\nabla_xU+\frac{\partial_{q}h(q)}{h(q)}\partial_{\lambda}U+\partial_{q}\left(2\log h(q)+\frac{2\Delta H(q)}{\rho_1h(q)}\frac{\lambda^2}{e^{\lambda}}\right)\nonumber\\
&+8\pi\partial_{q}R(x,q)\mid_{x=q}+O(|x-q|)+O(\lambda^2e^{-\lambda}).
\end{align}
Since $\nabla_xU$ is symmetric with respect to $q$ in $B_{r_0}(q)$,
\begin{align}
\label{7.42}
\int_{B_{r_0}(q)}\rho_1h(q)\,e^{U}(U+s-1)\nabla_xU=\int_{B_{r_0}(q)}\rho_1h(q)\,e^{U}
O(|x-q|+\lambda^2e^{-\lambda})\nabla_xU=O(1).
\end{align}
Hence, noting the fact that $\partial_{\lambda}U$ is bounded, we deduce
\begin{align}
\label{7.43}
\int_{B_{r_0}(q)}\rho_1h(q)(U+s-1)\partial_{q}(v_q-\overline{v}_q)=O(\lambda).
\end{align}
For the other terms in (\ref{5.13}), it is possible to get the following estimates:
\begin{align}
\label{7.44}
\int_{B_{r_0}(q)}\rho_1h(q)\,e^{U}\partial_{q}(v_q-\overline{v}_q)=O(1),
\end{align}
\begin{align}
\label{7.45}
\int_{B_{r_0}(q)}\rho_1h(q)\,e^{U}O(|x-q|)\partial_{q}(v_q-\overline{v}_q)=O(1),
\end{align}
\begin{align}
\label{7.46}
\int_{B_{r_0}(q)}\rho_1h(q)\,e^{U}\nabla H(q)\cdot(x-q)\nabla_xU=(-8\pi+O(\lambda \,e^{-\lambda}))\nabla H(q),
\end{align}
and, using $\nabla H(q)=O(\lambda e^{-\lambda})$,
\begin{align}
\label{7.47}
\int_{B_{r_0}(q)}\rho_1h(q)\,e^{U}\nabla H(q)\cdot(x-q)\partial_{q}(v_q-\overline{v}_q)=8\pi\nabla H(q)+O(\lambda e^{-\frac32\lambda}).
\end{align}
For the term which involves $\phi$, we have
\begin{align}
\label{7.48}
\int_{B_{r_0}(q)}\rho_1h(q)\,e^{U}\phi\,\partial_{q}(v_q-\overline{v}_q)=\int_{B_{r_0}(q)}\rho_1h(q)\,e^{U}\phi\,\partial_{q}U
+\int_{B_{r_0}(q)}\rho_1h(q)\,e^{U}\phi\left(O\left(\frac{\lambda^2}{e^{\lambda}}\right)+O(|x-q|)\right).
\end{align}
Using $\l\nabla\phi,\nabla\partial_{q}(v_q-\overline{v}_q)\r=0$ one gets
\begin{align*}
0=&\int_{M}\nabla\phi\nabla\partial_{q}v_q=-\int_{M}\phi\Delta(\partial_{q}v_q)\\
=&\int_{B_{r_0}(q)}\rho_1h(q)\,e^{U}\partial_{q}U\phi+
\partial_{q}\log h(q)\int_{B_{r_0}(q)}\rho_1h(q)\,e^{U}\phi+O(\lambda^2e^{-\lambda})\|\phi\|_{H^1(M)}.
\end{align*}
By (\ref{7.2}) and the above equality, we have
\begin{align*}
\int_{B_{r_0}(q)}2\rho_1h(q)\,e^{U}\partial_{q}U\phi=O(\lambda^2e^{-\lambda})\|\phi\|_{H^1(\Omega)}.
\end{align*}
On the other hand, for the terms concerning $\frac{e^t}{\int_Mh_1e^{v_1-v_2}}-1$ and $\psi$, we obtain
\begin{align*}
\int_{B_{r_0}(q)}\rho_1h(q)\,e^{U}\left(\frac{e^t}{\int_Mh_1e^{v_1-v_2}}-1-\psi\right)\partial_q(v_q-\overline{v}_q)
=-8\pi\nabla\psi(q)+O\left(\frac{e^t}{\int_Mh_1e^{v_1-v_2}}-1-\psi(q)\right),
\end{align*}
where we used
\begin{align*}
\int_{B_{r_0}(q)}\rho_1h(q)\,e^U\nabla\psi(q)(x-q)\nabla_yU=(8\pi+O(e^{-\lambda}))\nabla\psi(q)
\end{align*}
and (\ref{7.44}). We can see that $\partial_{q}(v_q-\overline{v}_q)=O(e^{\frac12\lambda}).$ Hence, as in the proof of part (3) we have
\begin{align*}
\int_{B_{r_0}(q)}E\partial_{\lambda_q}(v_q-\overline{v}_q)=O(\lambda^3e^{-\frac32\lambda}).
\end{align*}

\medskip

We consider now $M\setminus B_{r_0}(q).$ In this case $\partial_{q}(v_q-\overline{v}_q)=O(1).$ Hence by \cite[Lemma 2.2]{cl1} we get
$$\int_{B_{2r_0}(q)\setminus B_{r_0}(q)}\Delta(v_q-8\pi G(x,q))\cdot\partial_{q}(v_q-\overline{v}_q)=O(\lambda \,e^{-\lambda}).$$
Since $\int_{M}h_1e^{v_1-v_2}=O(e^{-\lambda}),$ it is not difficult to see that the integral of the products of $\partial_{q}v_q$ and the nonlinear terms in (\ref{5.14}) and (\ref{5.15}) are of order $O(e^{-\lambda}).$

\bigskip

The estimates above imply
\begin{align}
\label{7.49}
&\l\nabla(v_1+T_1(v_1,v_2)),\nabla\partial_{q}(v_q-\overline{v}_q)\r\nonumber\\
=&-8\pi\nabla H(q)+8\pi\nabla \psi(q)+O\left(\lambda |a-1|+\left|\frac{e^{t}}{\int_{M} h_1e^{v_1-v_2}}-1-\psi(q)\right|
+\frac{\lambda}{e^{\lambda}}\right).
\end{align}
This concludes the proof of part (2) and of the proposition.
\begin{flushright}
$\square$\\
\end{flushright}

\end{document}